\pdfoutput=1
\RequirePackage{ifpdf}
\ifpdf 
\documentclass[pdftex]{sigma}
\else
\documentclass{sigma}
\fi

\usepackage{bm}

\numberwithin{equation}{section}

\newtheorem{Theorem}{Theorem}[section]
\newtheorem*{Theorem*}{Theorem}
\newtheorem{Corollary}[Theorem]{Corollary}
\newtheorem{Lemma}[Theorem]{Lemma}
\newtheorem{Proposition}[Theorem]{Proposition}
 { \theoremstyle{definition}
\newtheorem{Definition}[Theorem]{Definition}

\newtheorem{Remark}[Theorem]{Remark} }

\begin{document}
\allowdisplaybreaks

\newcommand{\arXivNumber}{2102.09175}

\renewcommand{\PaperNumber}{080}

\FirstPageHeading

\ShortArticleName{Connection Problem for an Extension of $q$-Hypergeometric Systems}

\ArticleName{Connection Problem for an Extension\\ of $\boldsymbol{q}$-Hypergeometric Systems}

\Author{Takahiko NOBUKAWA}

\AuthorNameForHeading{T.~Nobukawa}

\Address{Department of Mathematics, Kobe University, Rokko, Kobe 657-8501, Japan}
\Email{\href{mailto:tnobukw@math.kobe-u.ac.jp}{tnobukw@math.kobe-u.ac.jp}}

\ArticleDates{Received March 19, 2021, in final form October 14, 2022; Published online October 21, 2022}

\Abstract{We give an example of solutions of the connection problem associated with a certain system of linear $q$-difference equations recently introduced by Park. The result contains a connection formulas of the $q$-Lauricella hypergeometric function $\varphi_{D}$ and those of the $q$-generalized hypergeometric function ${}_{N+1}\varphi_{N}$ as special cases.}


\Classification{33D70; 39A13}

\section{Introduction}\label{secintro}
In \cite{T}, a hypergeometric function $F_{L,N}$, which is a certain generalization of hypergeometric functions, was defined by Tsuda.
He also obtained a Hamiltonian system $\mathcal{H}_{L,N}$, which describes an isomonodromic deformation of an $L\times L$ Fuchsian system on $\mathbb{P}^{1}$ with $N+3$ regular singularities, and which has particular solutions in terms of the function $F_{L,N}$.
In \cite{P1,P2}, a $q$-analog of Tsuda's result was obtained by Park. Namely, she defined a $q$-hypergeometric function $\mathcal{F}_{N,M}$, which is given by~(\ref{ser1}) below, and a system $\mathcal{P}_{N,M}$ as a $q$-analog of the function $F_{N+1,M}$ and the system~$\mathcal{H}_{N+1,M}$, respectively.
The function $\mathcal{F}_{N,M}$ converges locally and satisfies linear $q$-difference equations, given by (\ref{eqn1}) and (\ref{eqn2}) below.
Also the function $\mathcal{F}_{N,M}$ has the Euler-type integral representation, given by~(\ref{JACKSON}).
Thus the function $\mathcal{F}_{N,M}$ must be continued analytically to $\mathbb{C}^{M}$.\looseness=-1

Our aim is to solve the connection problem related with the function $\mathcal{F}_{N,M}$, that is to give connection matrices among fundamental solutions of the $q$-difference equations~(\ref{eqn1}) and~(\ref{eqn2}).
The main result is Theorem \ref{thm1} in Section~\ref{secmat}.
Proposition~\ref{prop4} gives fundamental solutions for the equation, and Theorem~\ref{thm1} gives connection matrices among these fundamental solutions.
Our method to calculate the connection matrices is to use a connection formula of the generalized $q$-hypergeometric function $_{N+1}\varphi_{N}$ many times.
The connection matrices are given by the product of some matrices.

In the general theory of $q$-difference systems in several variables, local solutions should be characterized by the asymptotic behavior near the singularity in some prescribed sector.
For the general theory of $q$-difference systems in several variables, see~\cite{Ao1,Ao2}.
Our solutions are defined in the region $\{|t_{1}|\ll\cdots\ll|t_{L}|\ll1\ll|t_{L+1}|\ll\cdots\ll|t_{M}|\}$ where $0\leq L\leq M$, and characterized by the asymptotic behavior of the following form
\begin{gather*}
	t_{1}^{\delta_{1}}\cdots t_{M}^{\delta_{M}}(1+O(||x||)),
\end{gather*}
at $x=(t_{1}/t_{2},\dots,t_{L-1}/t_{L},t_{L},1/t_{L+1},t_{L+1}/t_{L+2},\dots,t_{M-1}/t_{M})=(0,\dots,0)$.
Here $\delta=(\delta_{1},\dots,\allowbreak \delta_{M})$ is a suitable parameter.
For more details, see Definition~\ref{defseries}, Proposition \ref{prop3}, Remark~\ref{remarkexp}.

The function $\mathcal{F}_{N,M}$ is a generalization both of the $q$-Lauricella function $\varphi_{D}$ and of the generalized $q$-hypergeometric function ${}_{N+1}\varphi_{N}$.
Thus our results contain solving the connection problem related with $\varphi_{D}$ and also related with ${}_{N+1}\varphi_{N}$.
The equality (\ref{wat1}) below can be considered as a part of the connection relations for ${}_{N+1}\varphi_{N}$, which was first studied by Thomae \cite{W} and was proved by Watson \cite{Wat}.
Also the $q$-Lauricella function has a Jackson integral representation of Jordan--Pochhammer type, and the connection problem related with the Jackson integral of Jordan--Pochhammer type was solved by Mimachi \cite{M}.
In differential case, fundamental solutions of the equation related with the Lauricella function $F_{D}$ were obtained by Gelfand--Kapranov--Zelevinsky \cite{GKZ} as an application of the theory of the GKZ hypergeometric function.
Our solutions are generalizations of the Gelfand--Kapranov--Zelevinsky's solutions (see Remark~\ref{remarkGKZ}).

The contents of this paper are as follows.
In Section~\ref{secpre}, we give notations and the properties of the function~$\mathcal{F}_{N,M}$ and the system of $q$-difference equations satisfied by~$\mathcal{F}_{N,M}$.
In Section~\ref{secconn}, we solve the connection problem of the system.
In Section~\ref{secsol}, we show fundamental solutions of the $q$-difference system.
In Section~\ref{secmat}, we give the matrices which connect the fundamental solution with the other fundamental solution.
In Section~\ref{secYB}, we obtain an elliptic solution of the Yang--Baxter equation as an application of Section~\ref{secmat}.

\section{Preliminaries}\label{secpre}

In this paper, we fix $q\in\mathbb{C}$ with $0<|q|<1$.
We use the following notations throughout the paper:
\begin{gather*}
	(a)_{\infty}=\prod_{k=0}^{\infty}\big(1-aq^{k}\big),\qquad
	(a)_{m}=\frac{(a)_{\infty}}{(aq^{m})_{\infty}},\\
	(a_{1},\dots,a_{n})_{m}=(a_{1})_{m}\cdots(a_{n})_{m},\qquad
	\theta(x)=(x,q/x)_{\infty}.
\end{gather*}
The symmetric group of degree $n$ is denoted by $\mathfrak{S}_{n}$.
In addition, ${}^{\rm T}A$ is the transpose of a~matrix~$A$.
We also use the notation $abc/defg$ for the fraction $(abc)/(defg)$.
Moreover, for a~multi-index $m=(m_{1},\dots,m_{M})$, we often use the notations
\begin{gather}
 |m|=\sum_{i=1}^{M}m_{i},\qquad
	\label{nl}
m(l)=\sum_{i=1}^{l}m_{i}-\sum_{i=l+1}^{M}m_{i},
\end{gather}
where $0\leq l\leq M$.
Here, the empty sum is considered to be 0.
\begin{Definition}[{\cite[Definition 2.1]{P1}}]
	We assume $c_{j}\notin q^{\mathbb{Z}_{\leq0}}=\{q^{n};\ n=0,-1,-2,\dots\}$ for $1\leq j\leq N$. We define the function $\mathcal{F}_{N,M}$ as
	\begin{gather}
		\label{ser1}
		\mathcal{F}_{N,M}\left(\!
		\begin{matrix}
			\{a_{j}\}_{1\leq j\leq N},\{b_{i}\}_{1\leq i\leq M}\\
			\{c_{j}\}_{1\leq j \leq N}	
		\end{matrix};\{t_{i}\}_{1\leq i\leq M}\!\right)=\sum_{m_{1},\dots,m_{M}\geq0}\prod_{j=1}^{N}\frac{(a_{j})_{|m|}}{(c_{j})_{|m|}}\prod_{i=1}^{M}\frac{(b_{i})_{m_{i}}}{(q)_{m_{i}}}\prod_{i=1}^{M}t_{i}^{m_{i}}.\!\!\!
	\end{gather}
	The series $\mathcal{F}_{N,M}$ converges in the region $|t_{i}|<1$.
\end{Definition}
When $N=1$ or $M=1$, we have
\begin{gather*}
	\mathcal{F}_{1,M}\left(
	\begin{matrix}
		a,\{b_{i}\}_{1\leq i\leq M}\\
		c	
	\end{matrix};\{t_{i}\}_{1\leq i\leq M}\right)=\varphi_{D}\left(\begin{matrix}
		a,\{b_{i}\}_{1\leq i\leq M}\\
		c	
	\end{matrix};\{t_{i}\}_{1\leq i\leq M}\right),\\
	\mathcal{F}_{N,1}\left(
	\begin{matrix}
		\{a_{j}\}_{1\leq j\leq N},b\\
		\{c_{j}\}_{1\leq j\leq N}	
	\end{matrix};t\right)={}_{N+1}\varphi_{N}\left(
	\begin{matrix}
		\{a_{j}\}_{1\leq j\leq N},\,b\\
		\{c_{j}\}_{1\leq j\leq N}	
	\end{matrix};t\right),
\end{gather*}
where $\varphi_{D}$ and ${}_{N+1}\varphi_{N}$ are the $q$-Lauricella function and the generalized $q$-hypergeometric function defined by
\begin{gather*}
	\varphi_{D}\left(\begin{matrix}
		a,\{b_{i}\}_{1\leq i\leq M}\\
		c	
	\end{matrix};\{t_{i}\}_{1\leq i\leq M}\right)=\sum_{m_{1},\dots,m_{M}\geq0}\frac{(a)_{|m|}}{(c)_{|m|}}\prod_{i=1}^{M}\frac{(b_{i})_{m_{i}}}{(q)_{m_{i}}}\prod_{i=1}^{M}t_{i}^{m_{i}},\\
	{}_{N+1}\varphi_{N}\left(
	\begin{matrix}
		\{a_{j}\}_{1\leq j\leq N},\,b\\
		\{c_{j}\}_{1\leq j\leq N}	
	\end{matrix};t\right)=\sum_{m\geq0}\prod_{j=1}^{N}\frac{(a_{j})_{m}}{(c_{j})_{m}}\cdot\frac{(b)_{m}}{(q)_{m}}t^{m}.
\end{gather*}
\begin{Proposition}[{\cite[Proposition 2.1]{P1}}]
	The series $\mathcal{F}_{N,M}$ satisfies the relation
	\begin{gather}
		\nonumber\mathcal{F}_{N,M}\left(
		\begin{matrix}
			\{a_{j}\}_{1\leq j\leq N},\{b_{i}\}_{1\leq i\leq M}\\
			\{c_{j}\}_{1\leq j\leq N}	
		\end{matrix};\{t_{i}\}_{1\leq i\leq M}\right)\\
		\label{dual1}
\qquad{} =\prod_{j=1}^{N}\frac{(a_{j})_{\infty}}{(c_{j})_{\infty}}\prod_{i=1}^{M}\frac{(b_{i}t_{i})_{\infty}}{(t_{i})_{\infty}}\cdot\mathcal{F}_{M,N}\left(
		\begin{matrix}
			\{t_{i}\}_{1\leq i\leq M},\{c_{j}/a_{j}\}_{1\leq j\leq N}\\
			\{b_{i}t_{i}\}_{1\leq i\leq M}	
		\end{matrix};\{a_{j}\}_{1\leq j\leq N}\right).
	\end{gather}
\end{Proposition}
\begin{Remark}
	When $N=1$, the relation (\ref{dual1}) reduces to
	\begin{gather*}
		\varphi_{D}\left(
		\begin{matrix}
			a,\{b_{i}\}_{1\leq i\leq M}\\
			c	
		\end{matrix};\{t_{i}\}_{1\leq i\leq M}\right)=\frac{(a)_{\infty}}{(c)_{\infty}}\prod_{i=1}^{M}\frac{(b_{i}t_{i})_{\infty}}{(t_{i})_{\infty}}\cdot{}_{M+1}\varphi_{M}\left(
		\begin{matrix}
			\{t_{i}\}_{1\leq i\leq M},c/a\\
			\{b_{i}t_{i}\}_{1\leq i\leq M}	
		\end{matrix};a\right),
	\end{gather*}
	which is relation~(4.1) of \cite{A}, a generalization of Heine's transformation for ${}_{2}\varphi_{1}$ \cite[equation~(1.4.1)]{A}.
\end{Remark}
It is well known that Heine's transformation can be interpreted as a $q$-integral form, so called a Jackson integral, of ${}_{2}\varphi_{1}$.
Similarly, the relation~(\ref{dual1}) can be rewritten in the Jackson integral form as follows.
\begin{Corollary}[{\cite[Corollary 2.1]{P1}}]
	With $a_{j}=q^{\alpha_{j}}$, the relation \eqref{dual1} can be rewritten as
	\begin{gather}
		\nonumber\mathcal{F}_{N,M}\left(
		\begin{matrix}
			\{q^{\alpha_{j}}\}_{1\leq j\leq N},\{b_{i}\}_{1\leq i\leq M}\\
			\{c_{j}\}_{1\leq j\leq N}	
		\end{matrix};\{t_{i}\}_{1\leq i\leq M}\right)\\
\qquad{} =\prod_{i=1}^{N}\frac{(q^{\alpha_{i}},c_{i}/q^{\alpha_{i}})_{\infty}}{(c_{i},q)_{\infty}}\cdot\int_{0}^{1}\cdots \int_{0}^{1}\prod_{j=1}^{N}\left\{\frac{z_{j}^{\alpha_{j}-1}}{1-q}\frac{(qz_{j})_{\infty}}{(c_{j}z_{j}/q^{\alpha_{j}})_{\infty}}\right\}\nonumber\\
\qquad\quad{}\times \prod_{i=1}^{M}\frac{(b_{i}t_{i}z_{1}\cdots z_{N})_{\infty}}{(t_{i}z_{1}\cdots z_{N})_{\infty}}{\rm d}_{q}z_{1}\cdots {\rm d}_{q}z_{N},\label{JACKSON}
	\end{gather}
	where the Jackson integral is defined as
	\begin{gather*}
		\int_{0}^{c}f(z){\rm d}_{q}z=c(1-q)\sum_{m\geq0}f(cq^{m})q^{m},
	\end{gather*}
	for $c\in\mathbb{C}$.
\end{Corollary}
\begin{Proposition}[{\cite[Proposition 2.2]{P1}}]
	The series
\[
\mathcal{F}=\mathcal{F}_{N,M}\left(
	\begin{matrix}
		\{a_{j}\}_{1\leq j\leq N},\{b_{i}\}_{1\leq i\leq M}\\
		\{c_{j}\}_{1\leq j\leq N}	
	\end{matrix};\{t_{i}\}_{1\leq i\leq M}\right)
\]
 satisfies the $q$-difference equations
	\begin{gather}
		\label{eqn1}
		\Biggl\{t_{s}\prod_{j=1}^{N}(1-a_{j}T)\cdot(1-b_{s}T_{s})-\prod_{j=1}^{N}\big(1-c_{j}q^{-1}T\big)\cdot(1-T_{s})\Biggr\}\mathcal{F}=0, \qquad 1\leq s\leq M,\\
		\label{eqn2}
		 \{t_{r}(1-b_{r}T_{r})(1-T_{s})-t_{s}(1-b_{s}T_{s})(1-T_{r})\}\mathcal{F}=0,\qquad 1\leq r<s\leq M,
	\end{gather}
	where $T_{s}$ is the $q$-shift operator $T_{s}f(t)=f(\dots,qt_{s},\dots)$ for the variable $t_{s}$ and $T=\prod_{s=1}^{M}T_{s}$.
\end{Proposition}
In this paper, we use the notation $E_{N,M}$ for the system of $q$-difference equations~(\ref{eqn1}) and~(\ref{eqn2}).
\begin{Theorem}[{\cite[Theorem 4.1]{P2}}]
	The rank of $E_{N,M}$ is $NM+1$.
\end{Theorem}

\section[Connection problem of the q-difference system E\_\{N,M\}]{Connection problem of the $\boldsymbol{q}$-difference system $\boldsymbol{E_{N,M}}$}\label{secconn}

In this section, we consider the connection problem of the system $E_{N,M}$.
First, we show fundamental solutions of the system $E_{N,M}$, which converge locally.
Second, we calculate connection matrices.
The connection formula can be calculated by using Thomae--Watson's formula~(\ref{wat1}) many times.
We suppose $a_{j}=q^{\alpha_{j}}$, $b_{i}=q^{\beta_{i}}$ and $c_{j}=q^{\gamma_{j}}$, $1\leq i\leq M$, $1\leq j\leq N$.

\subsection[Solutions of the q-difference system E\_\{N,M\}]{Solutions of the $\boldsymbol{q}$-difference system $\boldsymbol{E_{N,M}}$}\label{secsol}

In this subsection, we show fundamental solutions of the $q$-difference system $E_{N,M}$ which converge locally.
\begin{Definition}\label{defseries}
	We assume
	\begin{gather}\label{COND3}
		a_{j}/a_{k}, c_{j}/c_{k}, a_{j}/b_{i}\cdots b_{M}, c_{k}/b_{i}\cdots b_{M}\notin q^{\mathbb{Z}},
	\end{gather}
	for $1\leq i\leq M+1$, $1\leq j\neq k\leq N$.
	Here, $q^{D}=\{q^x;\, x\in D\}$.
	We define series $\mathcal{F}_{N,M}^{L}$, $\mathcal{F}_{N,M}^{L;k,l}$ and $\mathcal{G}_{N,M}^{L;k,l'}$ as
	\begin{gather}
		\nonumber\mathcal{F}_{N,M}^{L}\left(
		\begin{matrix}
			\{a_{j}\}_{1\leq j\leq N},\{b_{i}\}_{1\leq i\leq M}\\
			\{c_{j}\}_{1\leq j\leq N}	
		\end{matrix};\{t_{i}\}_{1\leq i\leq M}\right)\\
		\label{sum1}
\qquad{} =\sum_{m_{1},\dots,m_{M}\geq0}\prod_{j=1}^{N}\frac{(a_{j}/b_{L+1}\cdots b_{M})_{m(L)}}{(c_{j}/b_{L+1}\cdots b_{M})_{m(L)}}\prod_{i=1}^{M}\frac{(b_{i})_{m_{i}}}{(q)_{m_{i}}}\prod_{i=1}^{L}t_{i}^{m_{i}}\prod_{i=L+1}^{M}\left(\frac{q}{b_{i}t_{i}}\right)^{m_{i}},\\
		\mathcal{F}_{N,M}^{L;k,l}\left(
		\begin{matrix}
			\{a_{j}\}_{1\leq j\leq N},\{b_{i}\}_{1\leq i\leq N}\\
			\{c_{j}\}_{1\leq j\leq N}	
		\end{matrix};\{t_{i}\}_{1\leq i\leq M}\right)
=\sum_{m_{1},\dots,m_{M}\geq0}\Biggl\{\prod_{j=1}^{N}\frac{(qa_{k}/c_{j})_{m_{L+1}}}{(qa_{k}/a_{j})_{m_{L+1}}}\nonumber\\
\qquad{}\times \prod_{i=1}^{L}\frac{(b_{i})_{m_{i}}}{(q)_{m_{i}}}\prod_{i=L+1}^{l-1}\frac{(b_{i})_{m_{i+1}}}{(q)_{m_{i+1}}} \prod_{i=l+1}^{M}\frac{(b_{i})_{m_{i}}}{(q)_{m_{i}}}\frac{(a_{k}/b_{l+1}\cdots b_{M})_{m(l)}}{(qa_{k}/b_{l}\cdots b_{M})_{m(l)}}\nonumber\\
		\qquad {}\times\prod_{i=1}^{L}\left(\frac{qt_{i}}{b_{l}t_{l}}\right)^{m_{i}}\prod_{i=L+1}^{l-1}\left(\frac{qt_{i}}{b_{l}t_{l}}\right)^{m_{i+1}}\prod_{i=l+1}^{M}\left(\frac{b_{l}t_{l}}{b_{i}t_{i}}\right)^{m_{i}}\left(\prod_{j=1}^{N}\frac{c_{j}}{a_{j}}\cdot\frac{q}{b_{l}t_{l}}\right)^{m_{L+1}}\Biggr\},\label{sum2}\\
		\mathcal{G}_{N,M}^{L;k,l'}\left(
		\begin{matrix}
			\{a_{j}\}_{1\leq j\leq N},\{b_{i}\}_{1\leq i\leq M}\\
			\{c_{j}\}_{1\leq j\leq N}	
		\end{matrix};\{t_{i}\}_{1\leq i\leq M}\right)=\sum_{m_{1},\dots,m_{M}\geq0}\Biggl\{\prod_{j=1}^{N}\frac{(qa_{j}/c_{k})_{m_{L}}}{(qc_{j}/c_{k})_{m_{L}}}\nonumber\\
\qquad{}\times \prod_{i=1}^{l'-1}\frac{(b_{i})_{m_{i}}}{(q)_{m_{i}}}\prod_{i=l'+1}^{L}\frac{(b_{i})_{m_{i-1}}}{(q)_{m_{i-1}}}\prod_{i=L+1}^{M}\frac{(b_{i})_{m_{i}}}{(q)_{m_{i}}}\frac{(c_{k}/qb_{l'+1}\cdots b_{M})_{m(l'-1)}}{(c_{k}/b_{l'}\cdots b_{M})_{m(l'-1)}}\nonumber\\
\qquad\times\prod_{i=1}^{l'-1}\left(\frac{qt_{i}}{b_{l'}t_{l'}}\right)^{m_{i}}\prod_{i=l'+1}^{L}\left(\frac{b_{l'}t_{l'}}{b_{i}t_{i}}\right)^{m_{i-1}}\prod_{i=L+1}^{M}\left(\frac{b_{l'}t_{l'}}{b_{i}t_{i}}\right)^{m_{i}}\cdot\left(\frac{b_{l'}t_{l'}}{q}\right)^{m_{L}}\Biggr\},\label{sum3}
	\end{gather}
	where $0\leq L \leq M$, $1\leq k\leq N$, $L+1\leq l\leq M$ and $1\leq l'\leq L$.
	Here, as mentioned in preliminaries~(\ref{nl}),
	\begin{gather*}
		m(l)=\sum_{i=1}^{l}m_{i}-\sum_{i=l+1}^{M}m_{i}.
	\end{gather*}
\end{Definition}
We find that the series (\ref{sum1}) converges in
\[
\left\{|t_{i}|<1, \, 1\leq i\leq L,\, \left|\frac{c_{1}\cdots c_{N}q}{a_{1}\cdots a_{N}b_{i}t_{i}}\right|<1,\, L+1\leq i\leq M\right\},
\]
 the series (\ref{sum2}) converges in
 \[
 \left\{\left|\frac{c_{1}\cdots c_{N}q}{a_{1}\cdots a_{N}b_{l}t_{l}}\right|<1,\, \left|\frac{qt_{i}}{b_{l}t_{l}}\right|<1,\, 1\leq i\leq l-1,\, \left|\frac{qt_{l}}{b_{i}t_{i}}\right|<1,\, l+1\leq i\leq M\right\},
 \]
and the series (\ref{sum3}) converges in
\[
\left\{|t_{l}|<1,\, \left|\frac{qt_{i}}{b_{l}t_{l}}\right|<1,\, 1\leq i\leq l-1,\, \left|\frac{qt_{l}}{b_{i}t_{i}}\right|<1,\, l+1\leq i\leq M\right\}.
\]
\begin{Proposition}
	\label{prop3}
	We assume the conditions \eqref{COND3}.
	For $0\leq L\leq M$, the functions
	\begin{gather}
		\label{sol1}\prod_{i=L+1}^{M}t_{i}^{-\beta_{i}}
		\cdot\mathcal{F}_{N,M}^{L}\left(
		\begin{matrix}
			\{a_{j}\}_{1\leq j\leq N},\{b_{i}\}_{1\leq i\leq M}\\
			\{c_{j}\}_{1\leq j\leq N}	
		\end{matrix};\{t_{i}\}_{1\leq i\leq M}\right),\\
		\label{sol2}t_{l}^{1+\sum_{i=l+1}^{M}\beta_{i}-\gamma_{k}}
		\prod_{i=l+1}^{M}t_{i}^{-\beta_{i}}
		\cdot\mathcal{G}_{N,M}^{L;k,l}\left(
		\begin{matrix}
			\{a_{j}\}_{1\leq j\leq N},\{b_{i}\}_{1\leq i\leq M}\\
			\{c_{j}\}_{1\leq j\leq N}	
		\end{matrix};\{t_{i}\}_{1\leq i\leq M}\right),\\ \nonumber\hspace{8cm} 1\leq k\leq N,\ 1\leq l\leq L,\\
		\label{sol3}t_{l}^{-\alpha_{k}+\sum_{i=l+1}^{M}\beta_{i}}
		\prod_{i=l+1}^{M}t_{i}^{-\beta_{i}}
		\cdot\mathcal{F}_{N,M}^{L;k,l}\left(
		\begin{matrix}
			\{a_{j}\}_{1\leq j\leq N},\{b_{i}\}_{1\leq i\leq M}\\
			\{c_{j}\}_{1\leq j\leq N}	
		\end{matrix};\{t_{i}\}_{1\leq i\leq M}\right),\\ \nonumber\hspace{8cm} 1\leq k\leq N,\ L+1\leq l\leq M,
	\end{gather}
	satisfy the $q$-difference system $E_{N,M}$.
\end{Proposition}

\begin{proof}
	We can check them easily.
	Here, we check that the function (\ref{sol1}) satisfies the $q$-difference equations (\ref{eqn1}).
	First, we have
	\begin{gather*}
		T_{i}t_{i}^{\alpha}=q^{\alpha}t_{i}^{\alpha}T_{i},
	\end{gather*}
	as an operator.
	Thus we obtain
	\begin{gather*}
\left(\prod_{i=L+1}^{M}t_{i}^{-\beta_{i}}\right)^{-1}\left\{t_{s}\prod_{j=1}^{N}(1-a_{j}T)\cdot(1-b_{s}T_{s})-\prod_{j=1}^{N}\big(1-c_{j}q^{-1}T\big)\cdot(1-T_{s})\right\} \prod_{i=L+1}^{M}t_{i}^{-\beta_{i}}
		\\
		=\begin{cases}
			\displaystyle
			\left\{t_{s}\prod_{j=1}^{N}\left(1-\frac{a_{j}}{b_{L+1}\cdots b_{M}}T\right)\cdot(1-b_{s}T_{s})-\prod_{j=1}^{N}\left(1-\frac{c_{j}q^{-1}}{b_{L+1}\cdots b_{M}}T\right)\cdot(1-T_{s})\right\},\\
			\hspace{11cm} 1\leq s\leq L,\\
			\displaystyle
			\left\{t_{s}\prod_{j=1}^{N}\left(1-\frac{a_{j}}{b_{L+1}\cdots b_{M}}T\right)\cdot(1-T_{s})-\prod_{j=1}^{N}\left(1-\frac{c_{j}q^{-1}}{b_{L+1}\cdots b_{M}}T\right)\cdot\left(1-\frac{1}{b_{s}}T_{s}\right)\right\},\\
			\hspace{11cm} L+1\leq s\leq M,
		\end{cases}
	\end{gather*}
	as an operator.
	For $1\leq s\leq L$, we have
	\begin{gather*}
		\nonumber t_{s}\prod_{j=1}^{N}\left(1-\frac{a_{j}}{b_{L+1}\cdots b_{M}}T\right)\cdot(1-b_{s}T_{s})\cdot\mathcal{F}_{N,M}^{L}\left(
		\begin{matrix}
			\{a_{j}\}_{1\leq j\leq N},\{b_{i}\}_{1\leq i\leq M}\\
			\{c_{j}\}_{1\leq j\leq N}	
		\end{matrix};\{t_{i}\}_{1\leq i\leq M}\right)\\
		\nonumber
 =t_{s}\prod_{j=1}^{N}\left(1-\frac{a_{j}}{b_{L+1}\cdots b_{M}}T\right)\cdot(1-b_{s}T_{s})\\
		\qquad{} \times\sum_{m_{1},\dots,m_{M}\geq0}\prod_{j=1}^{N}\frac{(a_{j}/b_{L+1}\cdots b_{M})_{m(L)}}{(c_{j}/b_{L+1}\cdots b_{M})_{m(L)}}\prod_{i=1}^{M}\frac{(b_{i})_{m_{i}}}{(q)_{m_{i}}}\prod_{i=1}^{L}t_{i}^{m_{i}}\prod_{i=L+1}^{M}\left(\frac{q}{b_{i}t_{i}}\right)^{m_{i}}\\
		\nonumber=\sum_{m_{1},\dots,m_{M}\geq 0}t_{s}\prod_{j=1}^{N}\left(1-\frac{a_{j}q^{m(L)}}{b_{L+1}\cdots b_{M}}\right)\cdot\big(1-b_{s}q^{m_{s}}\big)\\
		\qquad{} \times\prod_{j=1}^{N}\frac{(a_{j}/b_{L+1}\cdots b_{M})_{m(L)}}{(c_{j}/b_{L+1}\cdots b_{M})_{m(L)}}\prod_{i=1}^{M}\frac{(b_{i})_{m_{i}}}{(q)_{m_{i}}}\prod_{i=1}^{L}t_{i}^{m_{i}}\prod_{i=L+1}^{M}\left(\frac{q}{b_{i}t_{i}}\right)^{m_{i}}\\
		\nonumber=\sum_{m_{1},\dots,m_{M}\geq0}\prod_{j=1}^{N}\frac{(a_{j}/b_{L+1}\cdots b_{M})_{m(L)+1}}{(c_{j}/b_{L+1}\cdots b_{M})_{m(L)}}\prod_{\substack{1\leq i\leq M\\i\neq s}}\frac{(b_{i})_{m_{i}}}{(q)_{m_{i}}}\nonumber\\
\qquad{}\times \frac{(b_{s})_{m_{s}+1}}{(q)_{m_{s}}}\prod_{\substack{1\leq i\leq L\\ i\neq s}}t_{i}^{m_{i}}\cdot t_{s}^{m_{s}+1}\prod_{i=L+1}^{M}\left(\frac{q}{b_{i}t_{i}}\right)^{m_{i}}\\
		\nonumber=\sum_{\substack{m_{1},\dots,m_{s-1},m_{s+1},\dots,m_{M}\geq0}}\sum_{m_{s}\geq-1}\Biggl\{\prod_{j=1}^{N}\frac{(a_{j}/b_{L+1}\cdots b_{M})_{m(L)+1}}{(c_{j}/b_{L+1}\cdots b_{M})_{m(L)+1}}\prod_{j=1}^{N}\left(1-\frac{c_{j}q^{-1}q^{m(L)+1}}{b_{L+1}\cdots b_{M}}\right)\\
		\nonumber\qquad\times\prod_{\substack{1\leq i\leq M\\i\neq s}}\frac{(b_{i})_{m_{i}}}{(q)_{m_{i}}}\cdot\frac{(b_{s})_{m_{s}+1}}{(q)_{m_{s}+1}}\big(1-q^{m_{s}+1}\big)\prod_{\substack{1\leq i\leq L\\ i\neq s}}t_{i}^{m_{i}}\cdot t_{s}^{m_{s}+1}\prod_{i=L+1}^{M}\left(\frac{q}{b_{i}t_{i}}\right)^{m_{i}}\Biggr\}\\
		\nonumber=\sum_{m_{1},\dots,m_{M}\geq 0}\prod_{j=1}^{N}\left(1-\frac{c_{j}q^{-1}q^{m(L)}}{b_{L+1}\cdots b_{M}}\right)\cdot\big(1-q^{m_{s}}\big)\\
		\qquad{}\times\prod_{j=1}^{N}\frac{(a_{j}/b_{L+1}\cdots b_{M})_{m(L)}}{(c_{j}/b_{L+1}\cdots b_{M})_{m(L)}}\prod_{i=1}^{M}\frac{(b_{i})_{m_{i}}}{(q)_{m_{i}}}\prod_{i=1}^{L}t_{i}^{m_{i}}\prod_{i=L+1}^{M}\left(\frac{q}{b_{i}t_{i}}\right)^{m_{i}}\\
		=\prod_{n=1}^{N}\left(1-\frac{c_{n}q^{-1}T}{b_{L+1}\cdots b_{M}}\right)\cdot(1-T_{s})\cdot\mathcal{F}_{N,M}^{L}\left(
		\begin{matrix}
			\{a_{j}\}_{1\leq j\leq N},\{b_{i}\}_{1\leq i\leq M}\\
			\{c_{j}\}_{1\leq j\leq N}	
		\end{matrix};\{t_{i}\}_{1\leq i\leq M}\right).
	\end{gather*}
	For $L+1\leq s\leq M$, we have
	\begin{gather*}
		\nonumber t_{s}\prod_{j=1}^{N}\left(1-\frac{a_{j}}{b_{L+1}\cdots b_{M}}T\right)\cdot(1-T_{s})\cdot\mathcal{F}_{N,M}^{L}\left(
		\begin{matrix}
			\{a_{j}\}_{1\leq j\leq N},\{b_{i}\}_{1\leq i\leq M}\\
			\{c_{j}\}_{1\leq j\leq N}	
		\end{matrix};\{t_{i}\}_{1\leq i\leq M}\right)\\
		\nonumber =\sum_{m_{1},\dots,m_{M}\geq0}t_{s}\prod_{j=1}^{N}\left(1-\frac{a_{j}q^{m(L)}}{b_{L+1}\cdots b_{M}}\right)\cdot\big(1-q^{-m_{s}}\big)\\
		\qquad{}\times\prod_{j=1}^{N}\frac{(a_{j}/b_{L+1}\cdots b_{M})_{m(L)}}{(c_{j}/b_{L+1}\cdots b_{M})_{m(L)}}\prod_{i=1}^{M}\frac{(b_{i})_{m_{i}}}{(q)_{m_{i}}}\prod_{i=1}^{L}t_{i}^{m_{i}}\prod_{i=L+1}^{M}\left(\frac{q}{b_{i}t_{i}}\right)^{m_{i}}\\
		\nonumber =\sum_{{m_{1},\dots,m_{s-1},m_{s+1},\dots,m_{M}\geq0}}\sum_{m_{s}\geq1}\Biggl\{\prod_{j=1}^{N}\frac{(a_{j}/b_{L+1}\cdots b_{M})_{m(L)+1}}{(c_{j}/b_{L+1}\cdots b_{M})_{m(L)+1}}\prod_{j=1}^{N}\left(1-\frac{c_{j}q^{-1}q^{m(L)+1}}{b_{L+1}\cdots b_{M}}\right)\\
		\qquad{}\times\prod_{\substack{1\leq i\leq M\\i\neq s}}\! \frac{(b_{i})_{m_{i}}}{(q)_{m_{i}}}\cdot\frac{(b_{s})_{m_{s}-1}}{(q)_{m_{s}-1}}\cdot\left(1-\frac{q^{-m_{s}+1}}{b_{s}}\right)\prod_{i=1}^{L}t_{i}^{m_{i}}\! \prod_{\substack{L+1\leq i\leq M\\i\neq s}}\! \left(\frac{q}{b_{i}t_{i}}\right)^{m_{i}}\cdot\left(\frac{q}{b_{s}t_{s}}\right)^{m_{s}-1}\Biggr\}\\
		\nonumber =\sum_{m_{1},\dots,m_{M}\geq0}\prod_{j=1}^{N}\left(1-\frac{c_{j}q^{-1}q^{m(L)}}{b_{L+1}\cdots b_{M}}\right)\cdot\left(1-\frac{q^{-m_{s}}}{b_{s}}\right)\\
		\qquad{}\times\prod_{j=1}^{N}\frac{(a_{j}/b_{L+1}\cdots b_{M})_{m(L)}}{(c_{j}/b_{L+1}\cdots b_{M})_{m(L)}}\prod_{i=1}^{M}\frac{(b_{i})_{m_{i}}}{(q)_{m_{i}}}\prod_{i=1}^{L}t_{i}^{m_{i}}\prod_{i=L+1}^{M}\left(\frac{q}{b_{i}t_{i}}\right)^{m_{i}}\\
		=\prod_{j=1}^{N}\left(1-\frac{c_{j}q^{-1}}{b_{L+1}\cdots b_{M}}T\right)\cdot\left(1-\frac{1}{b_{s}}T_{s}\right)\cdot\mathcal{F}_{N,M}^{L}\left(
		\begin{matrix}
			\{a_{j}\}_{1\leq j\leq N},\{b_{i}\}_{1\leq i\leq M}\\
			\{c_{j}\}_{1\leq j\leq N}	
		\end{matrix};\{t_{i}\}_{1\leq i\leq M}\right).
	\end{gather*}
	Similar to these calculations, we can check Proposition \ref{prop3} directly.
\end{proof}

\begin{Remark}\label{remarkGKZ}
	If $N=1$, the solutions are $q$-analogs of the solutions of the differential equation related with Lauricella function $F_{D}$, which was obtained by Gelfand--Kapranov--Zelevinsky \cite{GKZ}.
	More precisely, the function (\ref{sum1}), (\ref{sum2}), (\ref{sum3}) are $q$-analogs of the function
	\begin{gather*}
		F_{D,j}\left(\begin{matrix}
			\tilde{\alpha};\{\tilde{\beta}_{i}\}_{1\leq i\leq M}\\
			\tilde{\gamma}
		\end{matrix};\{x_{i}\}_{1\leq i\leq M}\right)=\sum_{m_{1},\dots,m_{M}\geq0}\frac{(\tilde{\alpha})_{-m(j-1)}}{(\tilde{\gamma})_{-m(j-1)}} \prod_{i=1}^{M}\frac{(\tilde{\beta}_{i})_{m_{i}}}{(1)_{m_{i}}}\prod_{i=1}^{M}x_{i}^{m_{i}},
	\end{gather*}
	by replacing some parameters and transforming variables.
	Only here, $(\tilde{\alpha})_{n}=\Gamma(\tilde{\alpha}+n)/\Gamma(\tilde{\alpha})$ for $n\in\mathbb{Z}$.
	The solutions of the equation related with $F_{D}$ are given by using the function $F_{D,j}$.
\end{Remark}
\begin{Remark}\label{remarkexp}
	For $0\leq L\leq M$, the functions (\ref{sol1}), (\ref{sol2}), (\ref{sol3}) converge in the region
	\begin{gather*}
		D=\left\{\!|t_{i}|<1, \, 1\leq i\leq L,\, \left|\prod_{j=1}^{N}\frac{c_{j}}{a_{j}}\cdot\frac{q}{b_{i}t_{i}}\right|<1,\, L+1\leq i\leq M,\, \left|\frac{qt_{i}}{b_{j}t_{j}}\right|<1,\, 1\leq i<j\leq M\!\right\},
	\end{gather*}
	simultaneously.
	We put $x_{i}=t_{i}/t_{i+1}$, $1\leq i<L$, $x_{L}=t_{L}$, $x_{L+1}=1/t_{L+1}$, $x_{i}=t_{i-1}/t_{i}$, $L+1<i\leq M$, then we have $t_{i}=x_{i}\cdots x_{L}$, $1\leq i\leq L$, $1/t_{i}=x_{L+1}\cdots x_{i}$, $L+1\leq i\leq M$ and
	\begin{gather*}
		\frac{t_{i}}{t_{j}}=
		\begin{cases}
			x_{i}\cdots x_{j-1}, &1\leq i<j\leq L,\\
			x_{i}\cdots x_{j}, &1\leq i\leq L< j\leq M,\\
			x_{i+1}\cdots x_{j}, &L<i<j\leq M.
		\end{cases}
	\end{gather*}
	Therefore the functions (\ref{sol1}), (\ref{sol2}), (\ref{sol3}) are solutions of the system $E_{N,M}$ in the region \mbox{$\{|x_{i}|\ll 1, \, 1\leq i\leq M\}$}, and have the following asymptotic behavior at $x=(0,\dots,0)$:
	\begin{gather*}
		\prod_{i=L+1}^{M}t_{i}^{-\beta_{i}}
		\cdot\mathcal{F}_{N,M}^{L}\left(
		\begin{matrix}
			\{a_{j}\}_{1\leq j\leq N},\{b_{i}\}_{1\leq i\leq M}\\
			\{c_{j}\}_{1\leq j\leq N}	
		\end{matrix};\{t_{i}\}_{1\leq i\leq M}\right)=\prod_{i=L+1}^{M}t_{i}^{-\beta_{i}}\cdot(1+O(||x||)),\\
		t_{l}^{1+\sum_{i=l+1}^{M}\beta_{i}-\gamma_{k}}
		\prod_{i=l+1}^{M}t_{i}^{-\beta_{i}}
		\cdot\mathcal{G}_{N,M}^{L;k,l}\left(
		\begin{matrix}
			\{a_{j}\}_{1\leq j\leq N},\{b_{i}\}_{1\leq i\leq M}\\
			\{c_{j}\}_{1\leq j\leq N}	
		\end{matrix};\{t_{i}\}_{1\leq i\leq M}\right)\\
	\qquad{}	=t_{l}^{1+\sum_{i=l+1}^{M}\beta_{i}-\gamma_{k}}
		\prod_{i=l+1}^{M}t_{i}^{-\beta_{i}}\cdot(1+O(||x||)),\\
		t_{l}^{-\alpha_{k}+\sum_{i=l+1}^{M}\beta_{i}}
		\prod_{i=l+1}^{M}t_{i}^{-\beta_{i}}
		\cdot\mathcal{F}_{N,M}^{L;k,l}\left(
		\begin{matrix}
			\{a_{j}\}_{1\leq j\leq N},\{b_{i}\}_{1\leq i\leq M}\\
			\{c_{j}\}_{1\leq j\leq N}	
		\end{matrix};\{t_{i}\}_{1\leq i\leq M}\right)\\
		\qquad{} =t_{l}^{-\alpha_{k}+\sum_{i=l+1}^{M}\beta_{i}}
		\prod_{i=l+1}^{M}t_{i}^{-\beta_{i}}\cdot(1+O(||x||)),
	\end{gather*}
	where $||x||=\sum_{i=1}^{M}|x_{i}|$.
	The system $E_{N,M}$ is rewritten as follows:
	\begin{gather*}
		\begin{cases}
			\displaystyle\left\{\frac{t_{s}}{t_{s+1}}\cdots\frac{t_{L-1}}{t_{L}}t_{L} \prod_{j=1}^{N}(1-a_{j}T)\cdot(1-b_{s}T_{s})-\prod_{j=1}^{N}\big(1-c_{j}q^{-1}T\big)\cdot(1-T_{s})\right\}\mathcal{F}=0,\\
			\hspace{11cm} 1\leq s\leq L,\\
			\displaystyle\left\{ \prod_{j=1}^{N}(1-a_{j}T)\cdot(1-b_{s}T_{s})-\frac{1}{t_{L+1}}\frac{t_{L+1}}{t_{L+2}}\cdots\frac{t_{s-1}}{t_{s}}\prod_{j=1}^{N}\big(1-c_{j}q^{-1}T\big)\cdot(1-T_{s})\right\}\mathcal{F}=0, \\
			\hspace{11cm} L+1\leq s\leq M,
		\end{cases}\\
\left\{\frac{t_{r}}{t_{s}}(1-b_{r}T_{r})(1-T_{s})-(1-b_{s}T_{s})(1-T_{r})\right\}\mathcal{F}=0,\qquad 1\leq r<s\leq M.
	\end{gather*}
	If a function $f(t)$, which is defined in the region $D$, is a solution of $E_{N,M}$ and has asymptotic behavior
	\begin{gather*}
		f(t)=t_{1}^{\delta_{1}}\cdots t_{M}^{\delta_{M}}(1+O(||x||))
	\end{gather*}
	at $x=(0,\dots,0)$.
	Then, by checking the coefficient of the lowest power of $x$, we find that $\delta=(\delta_{1},\dots,\delta_{M})$ must satisfy
	\begin{gather*}
		\begin{cases}
			\displaystyle\prod_{j=1}^{N}\big(1-c_{j}q^{-1}q^{\delta_{1}+\cdots+\delta_{M}}\big)\big(1-q^{\delta_{s}}\big)=0, & 1\leq s\leq L,\\
			\displaystyle\prod_{j=1}^{N}\big(1-a_{j}q^{\delta_{1}+\cdots+\delta_{M}}\big)\big(1-b_{s}q^{\delta_{s}}\big)=0, &L<s\leq M,
		\end{cases}\\
		\big(1-b_{s}q^{\delta_{s}}\big)\big(1-q^{\delta_{r}}\big)=0, \qquad 1\leq r<s\leq M.
	\end{gather*}
	In this sense, the solution $(\delta_{1},\dots,\delta_{M})$ of the above equations should be called ``the characteristic exponent of the system $E_{N,M}$ at $x=(0,\dots,0)$'', like a characteristic exponent in one variable case.
	Solving those equations, we have
	\begin{gather}\label{CE1}
 (\delta_{1},\dots,\delta_{M}) =(0,\dots,0,-\beta_{L+1},\dots,-\beta_{M}),\\
		\label{CE2}\begin{cases}
			\displaystyle \left(1+\sum_{i=2}^{M}\beta_{i}-\gamma_{k},-\beta_{2},\dots,-\beta_{M}\right),\\
			\displaystyle \left(0,1+\sum_{i=3}^{M}\beta_{i}-\gamma_{k},-\beta_{3},\dots,-\beta_{M}\right),\\
			\cdots\cdots\cdots\cdots\cdots\cdots\cdots\cdots\cdots\cdots\cdots\cdots\\
			\displaystyle\left(0,\dots,0,1+\sum_{i=L+1}^{M}\beta_{i}-\gamma_{k},-\beta_{L+1},\dots,-\beta_{M}\right),
		\end{cases}\\
		\label{CE3}\begin{cases}
			\displaystyle \left(0,\dots,0,-\alpha_{k}+\sum_{i=L+2}^{M}\beta_{i},-\beta_{L+2},\dots,-\beta_{M}\right),\\
			\displaystyle \left(0,\dots,0,-\alpha_{k}+\sum_{i=L+3}^{M}\beta_{i},-\beta_{L+3},\dots,-\beta_{M}\right),\\
\cdots\cdots\cdots\cdots\cdots\cdots\cdots\cdots\cdots\cdots\cdots\cdots\\
			(0,\dots,0,-\alpha_{k}),
		\end{cases}
	\end{gather}
	where $1\leq k\leq N$.
	The solutions (\ref{sol1}), (\ref{sol2}), (\ref{sol3}) correspond to ``the characteristic exponents'' (\ref{CE1}), (\ref{CE2}), (\ref{CE3}), respectively.
	Also the solutions (\ref{sol1}), (\ref{sol2}), (\ref{sol3}) are characterized by the asymptotic behavior
	\begin{gather*}
		\prod_{i=L+1}^{M}t_{i}^{-\beta_{i}}(1+O(||x||)),\qquad
		t_{l}^{1+\sum_{i=l+1}^{M}\beta_{i}-\gamma_{k}}
		\prod_{i=l+1}^{M}t_{i}^{-\beta_{i}}(1+O(||x||)),\\
		t_{l}^{-\alpha_{k}+\sum_{i=l+1}^{M}\beta_{i}}
		\prod_{i=l+1}^{M}t_{i}^{-\beta_{i}}(1+O(||x||)),
	\end{gather*}
	respectively at $x=(t_{1}/t_{2},\dots,t_{L-1}/t_{L},t_{L},1/t_{L+1},t_{L+1}/t_{L+2},\dots,t_{M-1}/t_{M})=(0,\dots,0)$.
\end{Remark}
We should check the linearly independence of the solutions (\ref{sol1}), (\ref{sol2}), (\ref{sol3}). The following lemma is useful to check the linearly independence.
\begin{Lemma}\label{lem3}
	For any $i\neq j$, we assume $\delta_{i}\neq \delta_{j}$.
	Here, $\delta_{i}=(\delta_{1,i},\dots,\delta_{M,i})$ and $\delta_{k,i}\in\mathbb{C}$ $1\leq i\leq n$, $1\leq k\leq M$.
	Then the functions
	\begin{gather*}
		f_{i}(t_{1},\dots,t_{M})=t^{\delta_{i}}(1+O(||t||))
	\end{gather*}
	are linearly independent on $K=\{C(t);\, \mbox{for any }i,\, T_{i}C(t)=C(t)\}$.
	Here, $t^{\delta_{i}}=t_{1}^{\delta_{1,i}}\cdots t_{M}^{\delta_{M,i}}$.
\end{Lemma}
\begin{proof}
	We assume that $C_{1}(t)f_{1}(t)+\cdots+C_{n}(t)f_{n}(t)=0$, where $C_{i}(t)\in K$, $1\leq i\leq n$.
	We can take $m=(m_{1},\dots,m_{M})\in\mathbb{Z}^{M}$ such that $(m,\delta_{i})\neq (m,\delta_{j})$ for any $i\neq j$, where $(m,\delta_{i})=m_{1}\delta_{1,i}+\cdots+m_{M}\delta_{M,i}$.
	We define the operator $R=T_{1}^{m_{1}}\cdots T_{M}^{m_{M}}$ for such $m$ and consider the following determinant
	\begin{gather*}
		D=\begin{vmatrix}
			C_{1}(t)f_{1}(t)& C_{2}(t)f_{2}(t)& C_{3}(t)f_{3}(t)&\cdots & C_{n}f_{n}(t)\\
			R(C_{1}(t)f_{1}(t)) &R(C_{2}(t)f_{2}(t)) &R(C_{3}(t)f_{3}(t))& \cdots & R(C_{n}(t)f_{n}(t))\\
			R^{2}(C_{1}(t)f_{1}(t)) &R^{2}(C_{2}(t)f_{2}(t)) &R^{2}(C_{3}(t)f_{3}(t))& \cdots & R^{2}(C_{n}(t)f_{n}(t))\\
			\vdots &\vdots & & &\vdots\\
			R^{n-1}(C_{1}(t)f_{1}(t)) &R^{n-1}(C_{2}(t)f_{2}(t)) &R^{n-1}(C_{3}(t)f_{3}(t))& \cdots & R^{n-1}(C_{n}(t)f_{n}(t))
		\end{vmatrix}.
	\end{gather*}
	By definition of $R$ and $C_{i}(t)$, we have $R^{k}(C_{i}(t)f_{i}(t))=C_{i}(t)R^{k}(f_{i}(t))$.
	Also, we have $R^{k}(f_{i}(t))=q^{(m,\delta_{i})k}t^{\delta_{i}}(1+O(||t||))$ by definition of $f_{i}(t)$.
	Thus we have
	\begin{gather*}
		D=C_{1}(t)\cdots C_{M}(t) t^{\delta_{1}}\cdots t^{\delta_{n}}\\
	\hphantom{D=}{}\times\begin{vmatrix}
			1 & 1 & 1 & \cdots & 1\\
			q^{(m,\delta_{1})} & q^{(m,\delta_{2})} &q^{(m,\delta_{3})} & \cdots & q^{(m,\delta_{n})}\\
			\big(q^{(m,\delta_{1})}\big)^{2} & \big(q^{(m,\delta_{2})}\big)^{2} &\big(q^{(m,\delta_{3})}\big)^{2} & \cdots & \big(q^{(m,\delta_{n})}\big)^{2}\\
			\vdots & \vdots &\vdots & &\vdots\\
			\big(q^{(m,\delta_{1})}\big)^{n-1} & \big(q^{(m,\delta_{2})}\big)^{n-1} &\big(q^{(m,\delta_{3})}\big)^{n-1} & \cdots & \big(q^{(m,\delta_{n})}\big)^{n-1}
		\end{vmatrix}(1+O(||t||)).
	\end{gather*}
	By the assumption $C_{1}(t)f_{1}(t)+\cdots+C_{n}(t)f_{n}(t)=0$, we have $D=0$.
	Also, by the condition $(m,\delta_{i})\neq (m,\delta_{j})$ $(i\neq j)$, the Vandermonde determinant is not $0$.
	Therefore we have that there exists $i$ such that $C_{i}(t)=0$.
	By induction on $n$, we have $C_{1}(t)=\cdots=C_{n}(t)=0$.
\end{proof}

By Lemma \ref{lem3}, we have that solutions (\ref{sol1}), (\ref{sol2}) (\ref{sol3}) are linearly independent on $K$ if parameters $\{a_{j}\}_{1\leq j\leq N}$, $\{b_{i}\}_{1\leq i\leq M}$, $\{c_{j}\}_{1\leq j\leq N}$ satisfy the condition~(\ref{COND3}).
Also, if a function $f\left(
\begin{matrix}
	\{a_{j}\}_{1\leq j\leq N},\{b_{i}\}_{1\leq i\leq M}\\
	\{c_{j}\}_{1\leq j\leq N}	
\end{matrix};\{t_{i}\}_{1\leq i\leq M}\right)$ satisfies the $q$-difference system $E_{N,M}$, then the function $f\left(
\begin{matrix}
	\{a_{j}\}_{1\leq j\leq N},\{b_{\sigma(i)}\}_{1\leq i\leq M}\\
	\{c_{j}\}_{1\leq j\leq N}	
\end{matrix};\{t_{\sigma(i)}\}_{1\leq i\leq M}\right)$ satisfies the same system for $\sigma\in\mathfrak{S}_{M}$.
Therefore we have the following proposition.
\begin{Proposition}	\label{prop4}
	For $0\leq L\leq M$ and $\sigma\in\mathfrak{S}_{M}$, we set
	\begin{gather*}
		u_{0}^{L,\sigma}=\prod_{i=L+1}^{M}t_{\sigma_{i}}^{-\beta_{\sigma_{i}}}
		\cdot\mathcal{F}_{N,M}^{L}\left(
		\begin{matrix}
			\{a_{j}\}_{1\leq j\leq N},\{b_{\sigma(i)}\}_{1\leq i\leq M}\\
			\{c_{j}\}_{1\leq j\leq N}	
		\end{matrix};\{t_{\sigma(i)}\}_{1\leq i\leq M}\right),\\
		u_{k,l}^{L,\sigma}=\begin{cases}
			\displaystyle t_{\sigma_{l}}^{1+\sum_{i=l+1}^{M}\beta_{\sigma_{i}}-\gamma_{k}}
			\prod_{i=l+1}^{M}t_{\sigma_{i}}^{-\beta_{\sigma_{i}}}
			\cdot\mathcal{G}_{N,M}^{L;k,l}\left(
			\begin{matrix}
				\{a_{j}\}_{1\leq j\leq N},\{b_{\sigma(i)}\}_{1\leq i\leq M}\\
				\{c_{j}\}_{1\leq j\leq N}	
			\end{matrix};\{t_{\sigma(i)}\}_{1\leq i\leq M}\right),\\
			\hspace{8cm} 1\leq k\leq N,\ 1\leq l\leq L,\\
			\displaystyle t_{\sigma_{l}}^{-\alpha_{k}+\sum_{i=l+1}^{M}\beta_{\sigma_{i}}}
			\prod_{i=l+1}^{M}t_{\sigma_{i}}^{-\beta_{\sigma_{i}}}
			\cdot\mathcal{F}_{N,M}^{L;k,l}\left(
			\begin{matrix}
				\{a_{j}\}_{1\leq j\leq N},\{b_{\sigma(i)}\}_{1\leq i\leq M}\\
				\{c_{j}\}_{1\leq j\leq N}	
			\end{matrix};\{t_{\sigma(i)}\}_{1\leq i\leq M}\right),\\
			\hspace{8cm} 1\leq k\leq N,\ L+1\leq l\leq M.
		\end{cases}
	\end{gather*}
	We set
	\begin{gather*}
		{\bm{u}}^{L,\sigma}={}^{\rm T}\big(u_{0}^{L,\sigma},u_{1,1}^{L,\sigma},\dots,u_{1,M}^{L,\sigma},u_{2,1}^{L,\sigma},\dots,u_{N,M}^{L,\sigma}\big).
	\end{gather*}
	Then ${\bm{u}}^{L,\sigma}$ is a fundamental solution of the $q$-difference system $E_{N,M}$ in the region $D^{L,\sigma}$ if parameters $\{a_{j}\}_{1\leq j\leq N}$, $\{b_{i}\}_{1\leq i\leq M}$, $\{c_{j}\}_{1\leq j\leq N}$ satisfy the condition
	\begin{gather}
		\label{CONDALL}a_{j}/a_{k}, c_{j}/c_{k}, a_{j}/b_{\sigma(i)}\cdots b_{\sigma(M)}, c_{k}/b_{\sigma(i)}\cdots b_{\sigma(M)}\notin q^{\mathbb{Z}},
	\end{gather}
	for $1\leq i\leq M+1$, $1\leq j\neq k\leq N$. Here,
	\begin{gather*}
		D^{L,\sigma}=\Biggl\{|t_{\sigma(i)}|<1,\, 1\leq i\leq L,\, \Bigg|\prod_{j=1}^{N}\frac{c_{j}}{a_{j}}\cdot\frac{q}{b_{\sigma(i)}t_{\sigma(i)}}\Bigg|<1,\, L+1\leq i\leq M,\\
\hphantom{D^{L,\sigma}=\Biggl\{}{}
\left|\frac{qt_{\sigma(i)}}{b_{\sigma(j)}t_{\sigma(j)}}\right|<1,\, 1\leq i< j\leq M\Biggr\}.
	\end{gather*}
\end{Proposition}

\subsection{Connection matrices}\label{secmat}
In this subsection, we take a fundamental solution $\bm{u}^{L,\sigma}$ on $D^{L,\sigma}$ and we consider the connection formula between $\bm{u}^{L_{1},\sigma_{1}}$ and $\bm{u}^{L_{2},\sigma_{2}}$, where $0\leq L_{1}, L_{2}\leq M$ and $\sigma_{1},\sigma_{2}\in\mathfrak{S}_{M}$.
We assume the condition~(\ref{CONDALL}) for $1\leq i\leq M+1$, $1\leq j\neq k\leq N$ and for any $\sigma\in\mathfrak{S}_{M}$.
We can solve this problem in principle by calculating the following matrices:
\begin{itemize}\itemsep=0pt
	\item the matrix which connects $\bm{u}^{L,\mathrm{id}}$ with $\bm{u}^{L+1,\mathrm{id}}$, $0\leq L\leq M-1$,
	\item the matrix which connects $\bm{u}^{L,\mathrm{id}}$ with $\bm{u}^{L-1,\mathrm{id}}$, $1\leq L\leq M$,
	\item the matrix which connects $\bm{u}^{M,\mathrm{id}}$ with $\bm{u}^{M,s_{r}}$, $0\leq r\leq M-1$,
\end{itemize}
where $s_{r}=(r,r+1)\in\mathfrak{S}_{M}$.
These matrices can be calculated by the Thomae--Watson's formula~\cite{W,Wat}.
\begin{Lemma}[\cite{W,Wat}]
	The connection formula of ${}_{N+1}\varphi_{N}$ is given as follows:
	\begin{gather}
		{}_{N+1}\varphi_{N}\left(
		\begin{cases}
			a_{1},\dots,a_{N+1}\\
			b_{1},\dots,b_{N}	
		\end{cases};t\right)\nonumber
		=\sum_{k=1}^{N+1}\Biggl\{\prod_{j=1}^{N}\frac{(b_{j}/a_{k})_{\infty}}{(b_{j})_{\infty}}\prod_{\substack{1\leq j\leq N+1\\j\neq k}}\frac{(a_{j})_{\infty}}{(a_{j}/a_{k})_{\infty}}\\
\qquad{} \times\frac{\theta(ta_{k})}{\theta(t)}{}_{N+1}\varphi_{N}\left(
		\begin{cases}
			\{qa_{k}/b_{j}\}_{1\leq j\leq N},\,a_{k}\\
			\{qa_{k}/a_{j}\}_{1\leq j\leq N+1,\,j\neq k}	
		\end{cases};\frac{b_{1}\cdots b_{N}q}{a_{1}\cdots a_{N+1}t}\right)\Biggr\}.\label{wat1}
	\end{gather}
\end{Lemma}
This formula can be derived by applying the Cauchy's residue theorem to the following integral:
\begin{gather*}
	\int_{C}\frac{(b_{1}x,\dots,b_{N}x,qx/t,t/x)_{\infty}}{(a_{1}x,\dots,a_{N+1}x,1/x)_{\infty}}\frac{{\rm d}x}{x}.
\end{gather*}
Here the contour $C$ is a deformation of the positively oriented unit circle so that the poles of $1/(a_{1}x,\dots,a_{N+1}x)_{\infty}$ lie outside $C$, and the poles of $1/(1/x)_{\infty}$ and 0 lie inside $C$.
For more details, see \cite[Section~4.10]{GR}.

First, we consider the matrix which connects $\bm{u}^{L,\mathrm{id}}$ with $\bm{u}^{L+1,\mathrm{id}}$.
If $l\neq L+1$, then we have
\begin{gather*}
	u_{k,l}^{L,\mathrm{id}}=u_{k,l}^{L+1,\mathrm{id}},
\end{gather*}
easily.
Hence, we should calculate the connection formula of $u_{0}^{L,\mathrm{id}}$, $u_{k,L+1}^{L,\mathrm{id}}$.
By rewriting the definition (\ref{sum1}), we have
\begin{gather}
	\nonumber\mathcal{F}_{N,M}^{L}\left(
	\begin{matrix}
		\{a_{j}\},\{b_{i}\}\\
		\{c_{j}\}	
	\end{matrix};\{t_{i}\}\right)\\
	\nonumber=\sum_{m_{1},\dots,m_{M}\geq0}\prod_{j=1}^{N}\frac{(a_{j}/b_{L+1}\cdots b_{M})_{m(L)}}{(c_{j}/b_{L+1}\cdots b_{M})_{m(L)}}\prod_{i=1}^{M}\frac{(b_{i})_{m_{i}}}{(q)_{m_{i}}}\prod_{i=1}^{L}t_{i}^{m_{i}}\prod_{i=L+1}^{M}\left(\frac{q}{b_{i}t_{i}}\right)^{m_{i}}\\
	\nonumber=\sum_{m_{1},\dots,m_{L},m_{L+2},\dots,m_{M}\geq0}\Biggl\{\prod_{j=1}^{N}\frac{(a_{j}/b_{L+1}\cdots b_{M})_{m(L)'}}{(c_{j}/b_{L+1}\cdots b_{M})_{m(L)'}}\prod_{\substack{1\leq i\leq M\\i\neq L+1}}\frac{(b_{i})_{m_{i}}}{(q)_{m_{i}}}\prod_{i=1}^{L}t_{i}^{m_{i}}\prod_{i=L+2}^{M}\left(\frac{q}{b_{i}t_{i}}\right)^{m_{i}}\\
	\qquad\times{}_{N+1}\varphi_{N}\left(
	\begin{matrix}
		\{qb_{L+1}\cdots b_{M}/c_{j}q^{m(L)'}\}_{1\leq j\leq N},\,b_{L+1}\\
		\{qb_{L+1}\cdots b_{M}/a_{j}q^{m(L)'}\}_{1\leq j\leq N}	
	\end{matrix};\prod_{j=1}^{N}\frac{c_{j}}{a_{j}}\cdot\frac{q}{b_{L+1}t_{L+1}}\right)\Biggr\}.\label{cal_F_NML}
\end{gather}
Here and in the following, we use the notation
\begin{gather*}
	m(l)'=\sum_{i=1}^{l}m_{i}-\sum_{i=l+2}^{M}m_{i},
\end{gather*}
for $m=(m_{1},\dots,m_{M}) $ and $0\leq l\leq M $.
By applying the formula (\ref{wat1}) to ${}_{N+1}\varphi_{N}$, we obtain
\begin{gather*}
	\nonumber\mathcal{F}_{N,M}^{L}\left(
	\begin{matrix}
		\{a_{j}\}_{1\leq j\leq N},\{b_{i}\}_{1\leq i\leq M}\\
		\{c_{j}\}_{1\leq j\leq N}	
	\end{matrix};\{t_{i}\}_{1\leq i\leq M}\right)\\
	\nonumber=\prod_{j=1}^{N}\frac{(qb_{L+2}\cdots b_{M}/a_{j},qb_{L+1}\cdots b_{M}/c_{j})_{\infty}}{(qb_{L+1}\cdots b_{M}/a_{j},qb_{L+2}\cdots b_{M}/c_{j})_{\infty}}\cdot\frac{\theta(t_{L+1}a_{1}\cdots a_{N}/c_{1}\cdots c_{N})}{\theta(t_{L+1}b_{L+1}a_{1}\cdots a_{N}/c_{1}\cdots c_{N})}\\
	\nonumber\qquad{} \times\mathcal{F}_{N,M}^{L+1}\left(
	\begin{matrix}
		\{a_{j}\}_{1\leq j\leq N},\{b_{i}\}_{1\leq i\leq M}\\
		\{c_{j}\}_{1\leq j\leq N}	
	\end{matrix};\{t_{i}\}_{1\leq i\leq M}\right)\\
	\nonumber\qquad {} +\sum_{d=1}^{N}\Biggl\{\prod_{j=1}^{N}\frac{(c_{d}/a_{j})_{\infty}}{(qb_{L+1}\cdots b_{M}/a_{j})_{\infty}}\prod_{\substack{1\leq j\leq N\\j\neq d}}\frac{(qb_{L+1}\cdots b_{M}/c_{j})_{\infty}}{(c_{d}/c_{j})_{\infty}}\cdot\frac{(b_{L+1})_{\infty}}{(c_{d}/qb_{L+2}\cdots b_{M})_{\infty}}\\
	\qquad{} \times\frac{\theta(t_{L+1}a_{1}\cdots a_{N}c_{d}/qb_{L+2}\cdots b_{M}c_{1}\cdots c_{N})}{\theta(t_{L+1}b_{L+1}a_{1}\cdots a_{N}/c_{1}\cdots c_{N})}\\
\qquad{}\times \mathcal{F}_{N,M}^{L+1;d,L+1}\left(
	\begin{matrix}
		\{a_{j}\}_{1\leq j\leq N},\{b_{i}\}_{1\leq i\leq M}\\
		\{c_{j}\}_{1\leq j\leq N}	
	\end{matrix};\{t_{i}\}_{1\leq i\leq M}\right)\Biggr\}.
\end{gather*}
By similar calculation of (\ref{cal_F_NML}), we have
\begin{gather*}
	\nonumber\mathcal{F}_{N,M}^{L;k,L+1}\left(
	\begin{matrix}
		\{a_{j}\}_{1\leq j\leq N},\{b_{i}\}_{1\leq i\leq M}\\
		\{c_{j}\}_{1\leq j\leq N}	
	\end{matrix};\{t_{i}\}_{1\leq i\leq M}\right)\\
	=\prod_{\substack{1\leq j\leq N\\j\neq k}}\frac{(qb_{L+2}\cdots b_{M}/a_{j})_{\infty}}{(qa_{k}/a_{j})_{\infty}}\cdot\frac{(q/b_{L+1})_{\infty}}{(qa_{k}/b_{L+1}\cdots b_{M})_{\infty}}\prod_{j=1}^{N}\frac{(qa_{k}/c_{j})_{\infty}}{(qb_{L+2}\cdots b_{M}/c_{j})_{\infty}}\\
\nonumber {}\times\frac{\theta(t_{L+1}b_{L+1}\cdots b_{M}a_{1}\cdots a_{N}/a_{k}c_{1}\cdots c_{N})}{\theta(t_{L+1}b_{L+1}a_{1}\cdots a_{N}/c_{1}\cdots c_{N})}\mathcal{F}_{N,M}^{L+1}\left(
	\begin{matrix}
		\{a_{j}\}_{1\leq j\leq N},\{b_{i}\}_{1\leq i\leq M}\\
		\{c_{j}\}_{1\leq j\leq N}	
	\end{matrix};\{t_{i}\}_{1\leq i\leq M}\right)\\
	\nonumber {} +\sum_{d=1}^{N}\Biggl\{\prod_{\substack{1\leq j\leq N\\j\neq k}}\frac{(c_{d}/a_{j})_{\infty}}{(qa_{k}/a_{j})_{\infty}}\cdot\frac{(c_{d}/b_{L+1}\cdots b_{M})_{\infty}}{(qa_{k}/b_{L+1}\cdots b_{M})_{\infty}}\prod_{\substack{1\leq j\leq N\\j\neq d}}\frac{(qa_{k}/c_{j})_{\infty}}{(c_{d}/c_{j})_{\infty}}\cdot\frac{(a_{k}/b_{L+2}\cdots b_{M})_{\infty}}{(c_{d}/qb_{L+2}\cdots b_{M})_{\infty}}\\
{}\times\frac{\theta(t_{L+1}b_{L+1}a_{1}\cdots a_{N}c_{d}/qc_{1}\cdots c_{N}a_{k})}{\theta(t_{L+1}b_{L+1}a_{1}\cdots a_{N}/c_{1}\cdots c_{N})}\mathcal{G}_{N,M}^{L+1;d,L+1}\left(
	\begin{matrix}
		\{a_{j}\}_{1\leq j\leq N},\{b_{i}\}_{1\leq i\leq M}\\
		\{c_{j}\}_{1\leq j\leq N}	
	\end{matrix};\{t_{i}\}_{1\leq i\leq M}\right)\Biggr\}.
\end{gather*}
Therefore, we set
\begin{gather*}
	A^{L,\mathrm{id}}=A^{L,\mathrm{id}}\left(
	\begin{matrix}
		\{a_{j}\}_{1\leq j\leq N},\{b_{i}\}_{1\leq i\leq M}\\
		\{c_{j}\}_{1\leq j\leq N}
	\end{matrix};t_{L+1}\right)=
	\begin{pmatrix}
		A_{0,0} & A_{0,1} & \cdots & A_{0,N}\\
		A_{1,0} & A_{1,1} & \cdots & A_{1,N}\\
		\vdots & \vdots & & \vdots\\
		A_{N,0} & A_{N,1} & \cdots & A_{N,N}\\
	\end{pmatrix},\\
	A_{0,0}=\prod_{j=1}^{N}\frac{(qb_{L+2}\cdots b_{M}/a_{j},qb_{L+1}\cdots b_{M}/c_{j})_{\infty}}{(qb_{L+1}\cdots b_{M}/a_{j},qb_{L+2}\cdots b_{M}/c_{j})_{\infty}}\cdot\frac{\theta(t_{L+1}a_{1}\cdots a_{N}/c_{1}\cdots c_{N})}{\theta(t_{L+1}b_{L+1}a_{1}\cdots a_{N}/c_{1}\cdots c_{N})}t_{L+1}^{-\beta_{L+1}},\\
	A_{0,d}=(A_{0,(1,d)},A_{0,(2,d)},\dots,A_{0,(M,d)})=(0,\dots,0,A_{0,(L+1,d)},0,\dots,0),\\
	\nonumber A_{0,(L+1,d)}=\prod_{j=1}^{N}\frac{(c_{d}/a_{j})_{\infty}}{(qb_{L+1}\cdots b_{M}/a_{j})_{\infty}}\prod_{\substack{1\leq j\leq N\\j\neq d}}\frac{(qb_{L+1}\cdots b_{M}/c_{j})_{\infty}}{(c_{d}/c_{j})_{\infty}}\cdot\frac{(b_{L+1})_{\infty}}{(c_{d}/qb_{L+2}\cdots b_{M})_{\infty}}\\
\hphantom{A_{0,(L+1,d)}=}{}
\times\frac{\theta(t_{L+1}a_{1}\cdots a_{N}c_{d}/qb_{L+2}\cdots b_{M}c_{1}\cdots c_{N})}{\theta(t_{L+1}b_{L+1}a_{1}\cdots a_{N}/c_{1}\cdots c_{N})}t_{L+1}^{-1-\sum_{i=L+1}^{M}\beta_{i}+\gamma_{d}},\\
	A_{k,0}={}^{\rm T}(A_{(1,k),0},A_{(2,k),0},\dots,A_{(N,k),0})={}^{\rm T}(0,\dots,0,A_{(L+1,k),0},0,\dots,0),\\
	\nonumber A_{(L+1,k),0}=\prod_{\substack{1\leq j\leq N\\j\neq k}}\frac{(qb_{L+2}\cdots b_{M}/a_{j})_{\infty}}{(qa_{k}/a_{j})_{\infty}}\cdot\frac{(q/b_{L+1})_{\infty}}{(qa_{k}/b_{L+1}\cdots b_{M})_{\infty}}\prod_{j=1}^{N}\frac{(qa_{k}/c_{j})_{\infty}}{(qb_{L+2}\cdots b_{M}/c_{j})_{\infty}}\\
\hphantom{A_{(L+1,k),0}=}{}
\times\frac{\theta(t_{L+1}b_{L+1}\cdots b_{M}a_{1}\cdots a_{N}/a_{k}c_{1}\cdots c_{N})}{\theta(t_{L+1}b_{L+1}a_{1}\cdots a_{N}/c_{1}\cdots c_{N})}t_{L+1}^{-\alpha_{k}+\sum_{i=L+2}^{M}\beta_{i}},\\
	A_{k,d}=
	\begin{pmatrix}
		I_{L} & O & O\\
		O & A_{(L+1,k),(L+1,d)} & O\\
		O & O & I_{M-L-1}
	\end{pmatrix},\\
	\nonumber A_{(L+1,k),(L+1,d)}=\prod_{\substack{1\leq j\leq N\\j\neq k}}\frac{(c_{d}/a_{j})_{\infty}}{(qa_{k}/a_{j})_{\infty}}\cdot\frac{(c_{d}/b_{L+1}\cdots b_{M})_{\infty}}{(qa_{k}/b_{L+1}\cdots b_{M})_{\infty}}\\
\hphantom{A_{(L+1,k),(L+1,d)}=}{}
\times \prod_{\substack{1\leq j\leq N\\j\neq d}}\frac{(qa_{k}/c_{j})_{\infty}}{(c_{d}/c_{j})_{\infty}}\cdot\frac{(a_{k}/b_{L+2}\cdots b_{M})_{\infty}}{(c_{d}/qb_{L+2}\cdots b_{M})_{\infty}}\\
\hphantom{A_{(L+1,k),(L+1,d)}=}{}
\times\frac{\theta(t_{L+1}b_{L+1}a_{1}\cdots a_{N}c_{d}/qc_{1}\cdots c_{N}a_{k})}{\theta(t_{L+1}b_{L+1}a_{1}\cdots a_{N}/c_{1}\cdots c_{N})}t_{L+1}^{-1-\alpha_{k}+\gamma_{d}},
\end{gather*}
where $1\leq k,d\leq N$, $I_{n}$ is the unit matrix of degree $n$ and $O$ is the null matrix.
Then we have
\begin{gather*}
	\bm{u}^{L,\mathrm{id}}=A^{L,\mathrm{id}}\bm{u}^{L+1,\mathrm{id}}.
\end{gather*}
Secondly, we consider the matrix which connects $\bm{u}^{L,\mathrm{id}}$ with $\bm{u}^{L-1,\mathrm{id}}$.
This can be calculated by a similar method for deriving the matrix $A^{L,\mathrm{id}}$.
We set
\begin{gather*}
	B^{L,\mathrm{id}}=B^{L,\mathrm{id}}\left(
	\begin{matrix}
		\{a_{j}\}_{1\leq j\leq N},\{b_{i}\}_{1\leq i\leq M}\\
		\{c_{j}\}_{1\leq j\leq N}
	\end{matrix};t_{L}\right)=
	\begin{pmatrix}
		B_{0,0} & B_{0,1} & \cdots & B_{0,N}\\
		B_{1,0} & B_{1,1} & \cdots & B_{1,N}\\
		\vdots & \vdots & & \vdots\\
		B_{N,0} & B_{N,1} & \cdots &B_{N,N}\\
	\end{pmatrix},\\
	B_{0,0}=\prod_{j=1}^{N}\frac{(a_{j}/b_{L+1}\cdots b_{M},c_{j}/b_{L}\cdots b_{M})_{\infty}}{(a_{j}/b_{L}\cdots b_{M},c_{j}/b_{L+1}\cdots b_{M})_{\infty}}\cdot\frac{\theta(t_{L}b_{L})}{\theta(t_{L})}t_{L}^{\beta_{L}},\\
	B_{0,d}=(B_{0,(1,d)},B_{0,(2,d)},\dots,B_{0,(M,d)})=(0,\dots,0,B_{0,(L,d)},0,\dots,0),\\
	B_{0,(L,d)}=\prod_{j=1}^{N}\frac{(c_{j}/a_{d})_{\infty}}{(c_{j}/b_{L+1}\cdots b_{M})_{\infty}}\prod_{\substack{1\leq j\leq N\\j\neq d}}\frac{(a_{j}/b_{L+1}\cdots b_{M})_{\infty}}{(a_{j}/a_{d})_{\infty}}\cdot\frac{(b_{L})_{\infty}}{(b_{L}\cdots b_{M}/a_{d})_{\infty}}\\
\hphantom{B_{0,(L,d)}=}{}
\times \frac{\theta(t_{L}a_{d}/b_{L+1}\cdots b_{M})}{\theta(t_{L})}t_{L}^{\alpha_{d}-\sum_{i=L+1}^{M}\beta_{i}},\\
	B_{k,0}={}^{\rm T}(B_{(1,k),0},B_{(2,k),0},\dots,B_{(M,k),0})={}^{\rm T}(0,\dots,0,B_{(L,k),0},0,\dots,0),\\
	\nonumber B_{(L,k),0}=\prod_{\substack{1\leq j\leq N\\j\neq k}}\frac{(c_{j}/b_{L}\cdots b_{M})_{\infty}}{(qc_{j}/c_{k})_{\infty}}\cdot\frac{(q/b_{L})_{\infty}}{(q^{2}b_{L+1}\cdots b_{M}/c_{k})_{\infty}}\prod_{j=1}^{N}\frac{(qa_{j}/c_{k})_{\infty}}{(a_{j}/b_{L}\cdots b_{M})_{\infty}}\\
\hphantom{B_{(L,k),0}=}{}
\times\frac{\theta(t_{L}qb_{L}\cdots b_{M}/c_{k})}{\theta(t_{L})}t_{L}^{1+\sum_{i=L}^{M}\beta_{i}-\gamma_{k}},\\
	B_{k,d}=
	\begin{pmatrix}
		I_{L-1} & O & O\\
		O &B_{(L,k),(L,d)} & O\\
		O & O & I_{M-L}
	\end{pmatrix},\\
	\nonumber B_{(L,k),(L,d)}=\prod_{\substack{1\leq j\leq N\\j\neq k}}\frac{(c_{j}/a_{d})_{\infty}}{(qc_{j}/c_{k})_{\infty}}\cdot\frac{(qb_{L+1}\cdots b_{M}/a_{d})_{\infty}}{(q^{2}b_{L+1}\cdots b_{M}/c_{k})_{\infty}}\prod_{\substack{1\leq j\leq N\\j\neq d}}\frac{(qa_{j}/c_{k})_{\infty}}{(a_{j}/a_{d})_{\infty}}\cdot\frac{(qb_{L}\cdots b_{M}/c_{k})_{\infty}}{(b_{L}\cdots b_{M}/a_{d})_{\infty}}\\
\hphantom{B_{(L,k),(L,d)}=}{}
\times\frac{\theta(t_{L}qa_{d}/c_{k})}{\theta(t_{L})}t_{L}^{1+\alpha_{d}-\gamma_{k}},
\end{gather*}
where $1\leq k,d\leq N$.
Then we have
\begin{gather*}
	\bm{u}^{L,\mathrm{id}}=B^{L,\mathrm{id}}\bm{u}^{L-1,\mathrm{id}}.
\end{gather*}
Finally, we consider the matrix which connects $\bm{u}^{M,\mathrm{id}}$ with $\bm{u}^{M,s_{r}}$.
We have
\begin{gather*}
	u_{0}^{M,s_{r}}=u_{0}^{M,\mathrm{id}},\qquad
	u_{k,l}^{M,s_{r}}=u_{k,l}^{M,\mathrm{id}},
\end{gather*}
easily if $l\neq r,r+1$.
Thus we should calculate the connection formula of $u_{k,r}^{M,s_{r}}$, $u_{k,r+1}^{M,s_{r}}$.
We have
\begin{gather*}
	\nonumber\mathcal{G}_{N,M}^{M;k,r}\left(
	\begin{matrix}
		\{a_{j}\}_{1\leq j\leq N},\{b_{s_{r}(i)}\}_{1\leq i\leq M}\\
		\{c_{j}\}_{1\leq j\leq N}	
	\end{matrix};\{t_{s_{r}(i)}\}_{1\leq i\leq M}\right)\\
\nonumber=\sum_{m_{1},\dots,m_{M}\geq0}\Biggl\{\prod_{k=1}^{N}\frac{(qa_{k}/c_{j})_{m_{M}}}{(qc_{k}/c_{j})_{m_{M}}}\prod_{i=1}^{r-1}\frac{(b_{i})_{m_{i}}}{(q)_{m_{i}}}\prod_{i=r+2}^{M}\frac{(b_{i})_{m_{i-1}}}{(q)_{m_{i-1}}}\cdot\frac{(b_{r})_{m_{r}}}{(q)_{m_{r}}}\frac{(c_{k}/qb_{r}b_{r+2}\cdots b_{M})_{m(r-1)}}{(c_{k}/b_{r}\cdots b_{M})_{m(r-1)}}\\
\qquad\nonumber{}\times\prod_{i=1}^{r-1}\left(\frac{qt_{i}}{b_{r+1}t_{r+1}}\right)^{m_{i}}\prod_{i=r+2}^{M}\left(\frac{b_{r+1}t_{r+1}}{b_{i}t_{i}}\right)^{m_{i-1}}\cdot\left(\frac{b_{r+1}t_{r+1}}{b_{r}t_{r}}\right)^{m_{r}}\left(\frac{b_{r+1}t_{r+1}}{q}\right)^{m_{M}}\Biggr\}\\
\nonumber=\sum_{m_{1},\dots,m_{r-1},m_{r+1},\dots,m_{M}\geq0}
\Biggl\{\prod_{k=1}^{N}\frac{(qa_{k}/c_{j})_{m_{M}}}{(qc_{k}/c_{j})_{m_{M}}}\prod_{i=1}^{r-1}\frac{(b_{i})_{m_{i}}}{(q)_{m_{i}}}\prod_{i=r+2}^{M}\frac{(b_{i})_{m_{i-1}}}{(q)_{m_{i-1}}}\\
	\qquad{}\times
	\frac{(c_{k}/qb_{r}b_{r+2}\cdots b_{M})_{m(r-1)'}}{(c_{k}/b_{r}\cdots b_{M})_{m(r-1)'}}\!\prod_{i=1}^{r-1}\! \left(\frac{qt_{i}}{b_{r+1}t_{r+1}}\right)^{m_{i}}\!
\prod_{i=r+2}^{M}\! \left(\frac{b_{r+1}t_{r+1}}{b_{i}t_{i}}\right)^{m_{i-1}}\! \cdot\left(\frac{b_{r+1}t_{r+1}}{q}\right)^{m_{M}}\\
	\qquad{}\times{}_{2}\varphi_{1}\left(
	\begin{matrix}
		b_{r},qb_{r}\cdots b_{M}/c_{k}q^{m(r-1)'}\\
		q^{2}b_{r}b_{r+2}\cdots b_{M}/c_{k}q^{m(r-1)'}
	\end{matrix};\frac{qt_{r+1}}{b_{r}t_{r}}\right)\Biggr\},
\end{gather*}
and by applying the formula (\ref{wat1}) to ${}_{2}\varphi_{1}$, we obtain
\begin{gather*}
	\nonumber\mathcal{G}_{N,M}^{M;k,r}\left(
	\begin{matrix}
		\{a_{j}\}_{1\leq j\leq N},\{b_{s_{r}(i)}\}_{1\leq i\leq M}\\
		\{c_{j}\}_{1\leq j\leq N}	
	\end{matrix};\{t_{s_{r}(i)}\}_{1\leq i\leq M}\right)\\
	\nonumber=\frac{(q/b_{r+1},b_{r})_{\infty}}{(q^{2}b_{r}b_{r+2}\cdots b_{M}/c_{k},c_{k}/qb_{r+1}\cdots b_{M})_{\infty}}\frac{\theta(t_{r}c_{k}/t_{r+1}qb_{r+1}\cdots b_{M})}{\theta(t_{r}b_{r}/t_{r+1})}\\
	\qquad{} \times\mathcal{G}_{N,M}^{M;k,r}\left(
	\begin{matrix}
		\{a_{j}\}_{1\leq j\leq N},\{b_{i}\}_{1\leq i\leq M}\\
		\{c_{j}\}_{1\leq j\leq N}	
	\end{matrix};\{t_{i}\}_{1\leq i\leq M}\right)\\
	\qquad{} +\frac{(q^{2}b_{r+2}\cdots b_{M}/c_{k},qb_{r}\cdots b_{M}/c_{k})_{\infty}}{(q^{2}b_{r}b_{r+2}\cdots b_{M}/c_{k},q_{r+1}\cdots b_{M}/c_{k})_{\infty}}\frac{\theta(t_{r}/t_{r+1})}{\theta(t_{r}b_{r}/t_{r+1})}\\
	\qquad{} \times\mathcal{G}_{N,M}^{M;j,r+1}\left(
	\begin{matrix}
		\{a_{j}\}_{1\leq j\leq N},\{b_{i}\}_{1\leq i\leq M}\\
		\{c_{j}\}_{1\leq j\leq N}	
	\end{matrix};\{t_{i}\}_{1\leq i\leq M}\right).
\end{gather*}
Similarly, we have
\begin{gather*}
	\nonumber\mathcal{G}_{N,M}^{M;k,r+1}\left(
	\begin{matrix}
		\{a_{j}\}_{1\leq j\leq N},\{b_{s_{r}(i)}\}_{1\leq i\leq M}\\
		\{c_{j}\}_{1\leq j\leq N}	
	\end{matrix};\{t_{s_{r}(i)}\}_{1\leq i\leq M}\right)\\
	\nonumber=\frac{(c_{k}/b_{r}\cdots b_{M},c_{k}/qb_{r+2}\cdots b_{M})_{\infty}}{(c_{k}/b_{r}b_{r+2}\cdots b_{M},c_{k}/qb_{r+1}\cdots b_{M})_{\infty}}\frac{\theta(t_{r}b_{r}/t_{r+1}b_{r+1})}{\theta(t_{r}b_{r}/t_{r+1})}\\
	\qquad{}\times\mathcal{G}_{N,M}^{M;k,r}\left(
	\begin{matrix}
		\{a_{j}\}_{1\leq j\leq N},\{b_{i}\}_{1\leq i\leq M}\\
		\{c_{j}\}_{1\leq j\leq N}	
	\end{matrix};\{t_{i}\}_{1\leq i\leq M}\right)\\
	\nonumber\qquad{} +\frac{(q/b_{r},b_{r+1})_{\infty}}{(c_{k}/b_{r}b_{r+2}\cdots b_{M},qb_{r+1}\cdots b_{M}/c_{k})_{\infty}}\frac{\theta(t_{r}qb_{r}b_{r+2}\cdots b_{M}/t_{r+1}c_{k})}{\theta(t_{r}b_{r}/t_{r+1})}\\
	\qquad{}\times\mathcal{G}_{N,M}^{M;k,r+1}\left(
	\begin{matrix}
		\{a_{j}\}_{1\leq j\leq N},\{b_{i}\}_{1\leq i\leq M}\\
		\{c_{j}\}_{1\leq j\leq N}	
	\end{matrix};\{t_{i}\}_{1\leq i\leq M}\right).
\end{gather*}
Therefore, we set
\begin{gather}
	\label{matS}S_{s_{r}}^{M,\mathrm{id}}=S_{s_{r}}^{M,\mathrm{id}}\left(
	\begin{matrix}
		\{b_{i}\}_{1\leq i\leq M}\\
		\{c_{j}\}_{1\leq j\leq N}
	\end{matrix};\frac{t_{r}}{t_{r+1}}\right)=
	\begin{pmatrix}
		1 & O & \cdots & \cdots & O\\
		O &S_{r}^{1} & O & \cdots & O\\
		\vdots & O & S_{r}^{2} & &\vdots \\
		\vdots & \vdots & &\ddots &O\\
		O & O&\cdots &O &S_{r}^{N}
	\end{pmatrix},\\
	\nonumber S_{r}^{k}=
	\begin{pmatrix}
		I_{r-1} & O & O & O\\
		O & S_{r,r}^{k} & S_{r,r+1}^{k} & O\\
		O & S_{r+1,r}^{k} & S_{r+1,r+1}^{k} & O\\
		O & O & O& I_{M-r-1}
	\end{pmatrix},\\
	\nonumber S_{r,r}^{k}=\frac{(q/b_{r+1},b_{r})_{\infty}}{(q^{2}b_{r}b_{r+2}\cdots b_{M}/c_{k},c_{k}/qb_{r+1}\cdots b_{M})_{\infty}}\\
\nonumber\hphantom{S_{r,r}^{k}=}{}
\times
\frac{\theta(t_{r}c_{k}/t_{r+1}qb_{r+1}\cdots b_{M})}{\theta(t_{r}b_{r}/t_{r+1})}\left(\frac{t_{r}}{t_{r+1}}\right)^{-1-\sum_{i=r}^{M}\beta_{i}+\gamma_{k}},\\
	\nonumber S_{r,r+1}^{k}=\frac{(q^{2}b_{r+2}\cdots b_{M}/c_{k},qb_{r}\cdots b_{M}/c_{k})_{\infty}}{(q^{2}b_{r}b_{r+2}\cdots b_{M}/c_{k},qb_{r+1}\cdots b_{M}/c_{k})_{\infty}}\frac{\theta(t_{r}/t_{r+1})}{\theta(t_{r}b_{r}/t_{r+1})}\left(\frac{t_{r}}{t_{r+1}}\right)^{-\beta_{r}},\\
	\nonumber S_{r+1,r}^{k}=\frac{(c_{k}/b_{r}\cdots b_{M},c_{k}/qb_{r+2}\cdots b_{M})_{\infty}}{(c_{k}/b_{r}b_{r+2}\cdots b_{M},c_{k}/qb_{r+1}\cdots b_{M})_{\infty}}\frac{\theta(t_{r}b_{r}/t_{r+1}b_{r+1})}{\theta(t_{r}b_{r}/t_{r+1})}\left(\frac{t_{r}}{t_{r+1}}\right)^{-\beta_{r+1}},\\
	\nonumber S_{r+1,r+1}^{k}=\frac{(q/b_{r},b_{r+1})_{\infty}}{(c_{k}/b_{r}b_{r+2}\cdots b_{M},qb_{r+1}\cdots b_{M}/c_{k})_{\infty}}\\
\nonumber\hphantom{S_{r+1,r+1}^{k}=}{}
\times
\frac{\theta(t_{r}qb_{r}b_{r+2}\cdots b_{M}/t_{r+1}c_{k})}{\theta(t_{r}b_{r}/t_{r+1})}\left(\frac{t_{r}}{t_{r+1}}\right)^{1+\sum_{i=r+2}^{M}\beta_{i}-\gamma_{k}},
\end{gather}
where $1\leq k\leq N$.
Then we have
\begin{gather*}
	\bm{u}^{M,s_{r}}=S_{s_{r}}^{M,\mathrm{id}}\bm{u}^{M,\mathrm{id}}.
\end{gather*}
Moreover, for $\sigma\in\mathfrak{S}_{M}$, we set
\begin{gather*}
	A^{L,\sigma}=A^{L,\mathrm{id}}\left(
	\begin{matrix}
		\{a_{j}\}_{1\leq j\leq N},\{b_{\sigma(i)}\}_{1\leq i\leq M}\\
		\{c_{j}\}_{1\leq j\leq N}
	\end{matrix};t_{\sigma(L+1)}\right),\\
	B^{L,\sigma}=B^{L,\mathrm{id}}\left(
	\begin{matrix}
		\{a_{j}\}_{1\leq j\leq N},\{b_{\sigma(i)}\}_{1\leq i\leq M}\\
		\{c_{j}\}_{1\leq j\leq N}
	\end{matrix};t_{\sigma(L)}\right),\\
	S_{s_{r}}^{M,\sigma}=S_{s_{r}}^{M,\mathrm{id}}\left(
	\begin{matrix}
		\{b_{\sigma(i)}\}_{1\leq i\leq M}\\
		\{c_{j}\}_{1\leq j\leq N}
	\end{matrix};\frac{t_{\sigma(r)}}{t_{\sigma(r+1)}}\right).
\end{gather*}
Then we have
\begin{gather*}
	\bm{u}^{L,\sigma}=A^{L,\sigma}\bm{u}^{L+1,\sigma},\qquad
	\bm{u}^{L,\sigma}=B^{L,\sigma}\bm{u}^{L-1,\sigma},\qquad
	\bm{u}^{M,s_{r}\sigma}=S_{s_{r}}^{M,\sigma}\bm{u}^{M,\sigma}.
\end{gather*}
Therefore, we obtain the following theorem:
\begin{Theorem}	\label{thm1}
	We assume the condition \eqref{CONDALL} for $1\leq i\leq M+1$, $1\leq j\neq k\leq N$ and for any $\sigma\in\mathfrak{S}_{M}$.
	For $0\leq L_{1},L_{2}\leq M$ and $\sigma_{1},\sigma_{2}\in\mathfrak{S}_{M}$, we have
	\begin{gather*}
		\bm{u}^{L_{2},\sigma_{2}}=A^{L_{2},\sigma_{2}}A^{L_{2}+1,\sigma_{2}}\cdots A^{M-1,\sigma_{2}}S_{s_{r_{1}}}^{M,s_{r_{2}}\cdots s_{r_{I}}\sigma_{1}}S_{s_{r_{2}}}^{M,s_{r_{3}}\cdots s_{r_{I}}\sigma_{1}}\cdots S_{s_{r_{I}}}^{M,\sigma_{1}}\\
\hphantom{\bm{u}^{L_{2},\sigma_{2}}=}{}
\times B^{M,\sigma_{1}}B^{M-1,\sigma_{1}}\cdots B^{L_{1}+1,\sigma_{1}}\bm{u}^{L_{1},\sigma_{1}},
	\end{gather*}
	if $\sigma_{2}=s_{r_{1}}\cdots s_{r_{I}}\sigma_{1}$, where $s_{r}=(r,r+1)\in\mathfrak{S}_{M}$.
\end{Theorem}

\begin{Remark}Each element of the matrices $A^{L,\sigma}$, $B^{L,\sigma}$, $S_{s_{r}}^{M,\sigma}$ is a pseudo constant, i.e.,
	\begin{gather*}
T_{s}A^{L,\sigma}=A^{L,\sigma},\qquad T_{s}B^{L,\sigma}=B^{L,\sigma},\qquad T_{s}S_{s_{r}}^{M,\sigma}=S_{s_{r}}^{M,\sigma}.
	\end{gather*}
\end{Remark}

\section{A solution of the Yang--Baxter equation}\label{secYB}
In this section, we assume $b_{i}=q^{\beta_{i}}$, $1\leq i\leq M$, and $c_{j}=q^{\gamma_{j}}$, $1\leq j\leq N$.
We obtain an elliptic solution of the Yang--Baxter equation as an application of Theorem~\ref{thm1}.
The functions
\begin{gather*}
	v_{0}^{\sigma}=\mathcal{F}_{N,M}\left(
	\begin{matrix}
		\{a_{j}\}_{1\leq j\leq N},\{b_{i}\}_{1\leq i\leq M}\\
		\{c_{j}\}_{1\leq j\leq N}
	\end{matrix};\{t_{i}\}_{1\leq i\leq M}\right),\\
v_{k,l}^{\sigma}=t_{\sigma(l)}^{1+\sum_{i=l}^{M}\beta_{\sigma(i)}-\gamma_{k}}\prod_{m=l+1}^{M}t_{\sigma(m)}^{-\beta_{\sigma(m)}}\cdot\mathcal{F}_{N,M}^{M;k,l}\left(\begin{matrix}
		\{a_{j}\}_{1\leq j\leq N},\{b_{\sigma(i)}\}_{1\leq i\leq M}\\
		\{c_{j}\}_{1\leq j\leq N}	
	\end{matrix};\{t_{\sigma(i)}\}_{1\leq i\leq M}\right),
\end{gather*}
where $1\leq k\leq N$, $1\leq l\leq M$ and $\sigma\in\mathfrak{S}_{M}$, are solutions of the $q$-difference system $E_{N,M}$ in the region
\begin{gather*}
	D^{M,\sigma}=\left\{|t_{i}|<1,\, 1\leq i\leq M,\, \left|\frac{qt_{\sigma(i)}}{b_{\sigma(j)}t_{\sigma(j)}}\right|<1,\, 1\leq i<j\leq M\right\}.
\end{gather*}
Similar to the calculation of the matrix (\ref{matS}), we set
\begin{gather*}
	\bm{v}^{\sigma}={}^{\rm T}\big(v_{0}^{\sigma},v_{1,1}^{\sigma},\dots,v_{1,M}^{\sigma},v_{2,1}^{\sigma},\dots,v_{N,M}^{\sigma}\big),
\end{gather*}
and we set
\begin{gather}
	\nonumber\tilde{S}_{s_{r}}^{\mathrm{id}}=\tilde{S}_{s_{r}}^{\mathrm{id}}\left(
	\begin{matrix}
		\{b_{i}\}_{1\leq i\leq M}\\
		\{c_{j}\}_{1\leq j\leq N}
	\end{matrix};\frac{t_{r}}{t_{r+1}}\right)=
	\begin{pmatrix}
		1 & O & \cdots & \cdots & O\\
		O &\tilde{S}_{r}^{1} & O & \cdots & O\\
		\vdots & O & \tilde{S}_{r}^{2} & &\vdots \\
		\vdots & \vdots & &\ddots &O\\
		O & O&\cdots &O &\tilde{S}_{r}^{N}
	\end{pmatrix},\\
	\label{defS}\tilde{S}_{r}^{k}=\tilde{S}_{r}\left(
	\begin{matrix}
		\{b_{i}\}_{1\leq i\leq M}\\
		c_{k}
	\end{matrix};\frac{t_{r}}{t_{r+1}}\right)=
	\begin{pmatrix}
		I_{r-1} & O & O & O\\
		O & \tilde{S}_{r,r}^{k} & \tilde{S}_{r,r+1}^{k} & O\\
		O & \tilde{S}_{r+1,r}^{k} & \tilde{S}_{r+1,r+1}^{k} & O\\
		O & O & O& I_{M-r-1}
	\end{pmatrix},\\
	\nonumber\tilde{S}_{r,r}^{k}=\frac{(q/b_{r+1},b_{r})_{\infty}}{(q^{2}b_{r}b_{r+2}\cdots b_{M}/c_{k},c_{k}/qb_{r+1}\cdots b_{M})_{\infty}}\\
\nonumber\hphantom{\tilde{S}_{r,r}^{k}=}{}
\times\frac{\theta(t_{r}c_{k}/t_{r+1}qb_{r+1}\cdots b_{M})}{\theta(t_{r}b_{r}/t_{r+1})}\left(\frac{t_{r}}{t_{r+1}}\right)^{-1-\sum_{i=r}^{M}\beta_{i}+\gamma_{k}},\\
	\nonumber\tilde{S}_{r,r+1}^{k}=\frac{(q^{2}b_{r+2}\cdots b_{M}/c_{k},qb_{r}\cdots b_{M}/c_{k})_{\infty}}{(q^{2}b_{r}b_{r+2}\cdots b_{M}/c_{k},qb_{r+1}\cdots b_{M}/c_{k})_{\infty}}\frac{\theta(t_{r}/t_{r+1})}{\theta(t_{r}b_{r}/t_{r+1})}\left(\frac{t_{r}}{t_{r+1}}\right)^{-\beta_{r}},\\
	\nonumber\tilde{S}_{r+1,r}^{k}=\frac{(c_{k}/b_{r}\cdots b_{M},c_{k}/qb_{r+2}\cdots b_{M})_{\infty}}{(c_{k}/b_{r}b_{r+2}\cdots b_{M},c_{k}/qb_{r+1}\cdots b_{M})_{\infty}}\frac{\theta(t_{r}b_{r}/t_{r+1}b_{r+1})}{\theta(t_{r}b_{r}/t_{r+1})}\left(\frac{t_{r}}{t_{r+1}}\right)^{-\beta_{r+1}},\\
	\nonumber\tilde{S}_{r+1,r+1}^{k}=\frac{(q/b_{r},b_{r+1})_{\infty}}{(c_{k}/b_{r}b_{r+2}\cdots b_{M},qb_{r+1}\cdots b_{M}/c_{k})_{\infty}}\\
\nonumber\hphantom{\tilde{S}_{r+1,r+1}^{k}=}{}
\times \frac{\theta(t_{r}qb_{r}b_{r+2}\cdots b_{M}/t_{r+1}c_{k})}{\theta(t_{r}b_{r}/t_{r+1})}\left(\frac{t_{r}}{t_{r+1}}\right)^{1+\sum_{i=r+2}^{M}\beta_{i}-\gamma_{k}},
\end{gather}
where $1\leq k\leq N$ and $s_{r}=(r,r+1)\in\mathfrak{S}_{M}$.
Then we have
\begin{gather*}
	\bm{v}^{s_{r}}=\tilde{S}_{s_{r}}^{\mathrm{id}}\bm{v}^{\mathrm{id}}.
\end{gather*}
In addition, we set
\begin{gather*}
	\tilde{S}_{s_{r}}^{\sigma}=\tilde{S}_{s_{r}}^{\mathrm{id}}\left(
	\begin{matrix}
		\{b_{\sigma(i)}\}_{1\leq i\leq M}\\
		\{c_{j}\}_{1\leq j\leq N}
	\end{matrix};\frac{t_{\sigma(r)}}{t_{\sigma(r+1)}}\right),
\end{gather*}
for $\sigma\in\mathfrak{S}_{M}$, and then we have
\begin{gather*}
	\bm{v}^{s_{r}\sigma}=\tilde{S}_{s_{r}}^{\sigma}\bm{v}^{\sigma}.
\end{gather*}
\begin{Remark}
	The matrices $\tilde{S}_{s_{r}}^{\mathrm{id}}$, $1\leq r\leq M-1$, depend only on $t_{r}/t_{r+1}$ and the parameters $\{b_{i}\}$, $\{c_{j}\}$.
\end{Remark}
By the braid relation $(r,r+2)=s_{r}s_{r+1}s_{r}=s_{r+1}s_{r}s_{r+1}$, we have
\begin{gather*}
\bm{v}^{(r,r+2)}=\tilde{S}_{s_{r}}^{s_{r+1}s_{r}}\tilde{S}_{s_{r+1}}^{s_{r}} \tilde{S}_{s_{r}}^{\mathrm{id}}\bm{v}^{\mathrm{id}}=\tilde{S}_{s_{r+1}}^{s_{r}s_{r+1}}\tilde{S}_{s_{r}}^{s_{r+1}}\tilde{S}_{s_{r+1}}^{\mathrm{id}}\bm{v}^{\mathrm{id}}.
\end{gather*}
In particular, we find that the matrices $\tilde{S}_{r}$ satisfy the Yang--Baxter equation
\begin{gather*}
	\nonumber\tilde{S}_{r}\left(
	\begin{matrix}
		\{b_{s_{r+1}s_{r}(i)}\}_{1\leq i\leq M}\\
		c_{k}
	\end{matrix};u\right)
	\tilde{S}_{r+1}\left(
	\begin{matrix}
		\{b_{s_{r}(i)}\}_{1\leq i\leq M}\\
		c_{k}
	\end{matrix};uv\right)
	\tilde{S}_{r}\left(
	\begin{matrix}
		\{b_{i}\}_{1\leq i\leq M}\\
		c_{k}
	\end{matrix};v\right)\\	
	\qquad {}=
	\tilde{S}_{r+1}\left(
	\begin{matrix}
		\{b_{s_{r}s_{r+1}(i)}\}_{1\leq i\leq M}\\
		c_{k}
	\end{matrix};v\right)
	\tilde{S}_{r}\left(
	\begin{matrix}
		\{b_{s_{r+1}(i)}\}_{1\leq i\leq M}\\
		c_{k}
	\end{matrix};uv\right)
	\tilde{S}_{r+1}\left(
	\begin{matrix}
		\{b_{i}\}_{1\leq i\leq M}\\
		c_{k}
	\end{matrix};u\right),
\end{gather*}
where $u=t_{r+1}/t_{r+2}$, $v=t_{r}/t_{r+1}$.
\begin{Remark}
	For the details of the Yang--Baxter equation, see Jimbo's text \cite{J}.
\end{Remark}
\begin{Remark}
	Aomoto, Kato and Mimachi \cite{AKM} obtained an elliptic solution of the Yang--Baxter equation by considering the connection matrices of a holonomic $q$-difference system which was studied in \cite{M}.
	They obtained that the matrices
	\begin{gather*}
		P_{i}(u)=
		\begin{pmatrix}
			I_{i-1} & O & O\\
			O & W(\alpha'+(i-1)\beta',\beta';u) & O\\
			O & O & I_{n-i-1}
		\end{pmatrix},\\
		\nonumber W(\alpha,\beta;u)=
		\begin{pmatrix}
			\displaystyle u^{\alpha+3\beta+1}\frac{\theta\big(q^{-\beta}\big)\theta\big(uq^{\alpha+2\beta+1}\big)}{\theta\big(q^{-\alpha-2\beta}\big)\theta\big(uq^{-\beta}\big)}
			& \displaystyle q^{\beta+1}u^{\beta}\frac{\theta(u)\theta\big(q^{-\alpha-\beta+1}\big)\theta\big(q^{\alpha+3\beta+2}\big)}{\theta\big(q^{-\alpha-2\beta}\big)^{2}\theta\big(uq^{-\beta}\big)}\vspace{1mm}\\
			\displaystyle
			\displaystyle u^{\beta}\frac{\theta(u)}{\theta\big(uq^{-\beta}\big)}
			& \displaystyle u^{-\alpha-\beta}\frac{\theta\big(q^{-\beta}\big)\theta\big(uq^{-\alpha-2\beta}\big)}{\theta\big(q^{-\alpha-2\beta}\big)\theta\big(uq^{-\beta}\big)},
		\end{pmatrix},
	\end{gather*}
	where $1\leq i\leq n-1$, satisfy the Yang--Baxter equation
	\begin{gather*}
		P_{i}(u)P_{i+1}(uv)P_{i}(v)=P_{i+1}(v)P_{i}(uv)P_{i+1}(u).
	\end{gather*}
	They also found that the matrix $W$ is identified as the matrix $W'$ with entires of Boltzmann weight $\begin{matrix}
		\kappa\vspace{-1mm}\\
		\mu\, \Box\, \sigma \vspace{-2mm}\\
		\nu
	\end{matrix}$ of the $A_{1}^{(1)}$ face model discussed by Jimbo, Miwa and Okado \cite{JMO}.
	The matrix $W'$ is expressed as
	\begin{gather}
		\label{MATRIXJMO}W'=\begin{pmatrix}
			\displaystyle\frac{[a-u]}{[a]} &\displaystyle\frac{[u][a+1][a-1]}{[1][a]^{2}}\vspace{1mm}\\
			\displaystyle\frac{[u]}{[1]} & \displaystyle\frac{[a+u]}{[a]}
		\end{pmatrix},
	\end{gather}
	where $[u]=\theta_{1}(\pi u/L,q)$, $L\neq0$ is an arbitrary complex parameter and
	\begin{gather*}
		\theta_{1}(u,q)=2q^{1/8}\sin u\prod_{k=1}^{\infty}\big(1-2q^{k}\cos 2u+q^{2k}\big)\big(1-q^{k}\big)\\
\hphantom{\theta_{1}(u,q)}{}
=2q^{1/8}\sin u\big({\rm e}^{2\sqrt{-1}u}q,{\rm e}^{-2\sqrt{-1}u}q,q\big)_{\infty}.
	\end{gather*}
	The matrix $W$ is equivalent to $W'$ as follows:
	\begin{gather*}
		x^{2g_{c}}\begin{pmatrix}
			x^{-g_{a+1}} & 0\\
			0 & x^{-g_{a-1}}
		\end{pmatrix}W'\begin{pmatrix}
			x^{-g_{a+1}} & 0\\
			0 & x^{-g_{a-1}}
		\end{pmatrix}=\frac{\theta\big(xq^{-\beta}\big)}{\theta\big(q^{-\beta}\big)}W(\alpha,\beta;x),
	\end{gather*}
	with ${\rm e}^{2\pi\sqrt{-1}u/L}=x$, ${\rm e}^{2\pi\sqrt{-1}/L}=q^{\beta+1}$, ${\rm e}^{2\pi\sqrt{-1}a/L}=q^{-\alpha-2\beta}$, $g_{a-1}=\frac{\alpha+\beta}{2}$, $g_{a+1}=\frac{-\alpha-3\beta}{2}$, $2g_{c}=\frac{1}{2}$.
	On the other hand, by specializing parameters $b_{i}$ as $b_{1}=\cdots=b_{M}=q^{\beta}$, the matrices $\tilde{S}_{r}(u)=\tilde{S}_{r}\left(
	\begin{matrix}
		q^{\beta}\\
		c_{k}
	\end{matrix};u\right)$ satisfy the Yang--Baxter equation
	\begin{gather*}
		\tilde{S}_{r}(u)\tilde{S}_{r+1}(uv)\tilde{S}_{r}(v)=\tilde{S}_{r+1}(v)\tilde{S}_{r}(uv)\tilde{S}_{r+1}(u).
	\end{gather*}
	By definition (\ref{defS}), we have
	\begin{gather}
\tilde{S}_{r}(u)=
		\begin{pmatrix}
			I_{r-1} & O & O\\
			O & \tilde{W}(\gamma_{k}-2-(M-r-2)\beta,-\beta;u) & O\\
			O & O & I_{M-r-1}
		\end{pmatrix},\label{WMATRIX1}\\
\nonumber\tilde{W}(\alpha,\beta;u)=
		\begin{pmatrix}
			\displaystyle u^{\alpha+3\beta+1}\frac{\theta\big(q^{-\beta}\big)\theta\big(uq^{\alpha+2\beta+1}\big)}{\theta\big(q^{-\alpha-2\beta}\big)\theta\big(uq^{-\beta}\big)}
			&\displaystyle
u^{\beta}\frac{\theta(u)\big(q^{-\alpha-\beta},q^{-\alpha-3\beta-1}\big)_{\infty}}{\theta\big(uq^{-\beta}\big)\big(q^{-\alpha-2\beta},q^{-\alpha-2\beta-1}\big)_{\infty}}\vspace{1mm}\\
			\displaystyle
u^{\beta}\frac{\theta(u)\big(q^{\alpha+3\beta+2},q^{\alpha+\beta+1}\big)_{\infty}}{\theta\big(uq^{-\beta}\big)\big(q^{\alpha+2\beta+2},q^{\alpha+2\beta+1}\big)_{\infty}}
			&\displaystyle
			u^{-\alpha-\beta-1}\frac{\theta\big(uq^{-\alpha-2\beta-1}\big)\theta\big(q^{-\beta}\big)}{\theta\big(uq^{-\beta}\big)\theta\big(q^{-\alpha-2\beta-1}\big)}
		\end{pmatrix},
	\end{gather}
	and by means of easy calculations, we find that the matrices $W$ and $\tilde{W}$ are conjugate as follows:
	\begin{gather*}
		W(\alpha,\beta;u)=A(\alpha,\beta)^{-1}\tilde{W}(\alpha,\beta;u)A(\alpha,\beta)=B(\alpha,\beta)\tilde{W}(\alpha,\beta;u)B(\alpha,\beta)^{-1},
	\end{gather*}
	where
	\begin{gather*}
		A(\alpha,\beta)=
		\begin{pmatrix}
			1 & 0\\
			0 & f(\alpha,\beta)
		\end{pmatrix},\qquad
		B(\alpha,\beta)=
		\begin{pmatrix}
			f(\alpha,\beta) & 0\\
			0 & 1
		\end{pmatrix},\\
		f(\alpha,\beta)=\frac{\big(q^{\alpha+3\beta+2},q^{\alpha+\beta+1}\big)_{\infty}}{\big(q^{\alpha+2\beta+2},q^{\alpha+2\beta+1}\big)_{\infty}}.
	\end{gather*}
	In conclusion, our matrix $\tilde{W}$ is identified as Jimbo, Miwa and Okado's matrix $W'$ and with $n=M$, $\alpha'=\gamma_{k}-2-(M-3)\beta$ and $\beta'=-\beta$, our matrices $\tilde{S}_{r}(u)$ and Aomoto, Kato and Mimachi's matrices $P_{r}(u)$ are conjugate.
\end{Remark}

\section{Summary and discussion}\label{secsummary}
A summary of our results is as follows.
The main result of this paper is Theorem \ref{thm1}, which gives the connection matrices for fundamental solutions~$\bm{u}^{L,\sigma}$ of the $q$-difference system~$E_{N,M}$ (\ref{eqn1}) and (\ref{eqn2}).
The fundamental solution~$\bm{u}^{L,\sigma}$ converges in the region $\{|t_{\sigma(1)}|\ll\cdots\ll|t_{\sigma(L)}|\ll1\ll|t_{\sigma(L+1)}|\ll\cdots\ll|t_{\sigma(M)}|\}$.
The component of the fundamental solution $\bm{u}^{L,\mathrm{id}}$ has the asymptotic behavior of the form $t_{1}^{\delta_{1}}\cdots t_{M}^{\delta_{M}}(1+O(||x||))$ at $x=(t_{1}/t_{2},\dots,t_{L-1}/t_{L},t_{L},1/t_{L+1},t_{L+1}/t_{L+2},\dots,\allowbreak t_{M-1}/t_{M})=(0,\dots,0)$ for some $\delta$.
More precisely, see Remark~\ref{remarkexp}.
The way to get the connection matrices is to calculate ``easy'' connection matrices many times, i.e.,
\begin{gather*}
	\bm{u}^{L_{1},\sigma_{1}}\to\bm{u}^{L_{1}+1,\sigma_{1}}\to\cdots\to\bm{u}^{M,\sigma_{1}}\\
	\hphantom{\bm{u}^{L_{1},\sigma_{1}}}{} \to\bm{u}^{M,s_{r_{1}}\sigma_{1}}\to\bm{u}^{M,s_{r_{2}}s_{r_{1}}\sigma_{1}}\to\cdots\to\bm{u}^{M,s_{r_{I}}\cdots s_{r_{1}}\sigma_{1}}=\bm{u}^{M,\sigma_{2}}\\
\hphantom{\bm{u}^{L_{1},\sigma_{1}}}{} \to\bm{u}^{M-1,\sigma_{2}}\to\bm{u}^{M-2,\sigma_{2}}\to\cdots\to\bm{u}^{L_{2},\sigma_{2}}.
\end{gather*}
Here, $s_{r}=(r,r+1)\in \mathfrak{S}_{M}$.
Each step can be calculated by using Thomae--Watson's formula~\mbox{\cite{W,Wat}}, which is a connection formula of the function ${}_{N+1}\varphi_{N}$.
In addition, as an application of Theorem \ref{thm1}, we obtained a solution of the Yang--Baxter equation by considering the connection matrix between $\bm{u}^{M,(r,r+2)}$ and $\bm{u}^{M,\mathrm{id}}$.
Also we showed that our matrix~(\ref{WMATRIX1}) is identified as Jimbo, Miwa, Okado's matrix~(\ref{MATRIXJMO}), and our solution and Aomoto, Kato, Mimachi's solution~\cite{AKM} are conjugate.

There are many problems related to our results.
We mention four of them here.
\begin{itemize}\itemsep=0pt
	\item[(i)]By taking the limit $q\to1$ with $a_{j}=q^{\alpha_{j}}$, $b_{i}=q^{\beta_{i}}$, $c_{j}=q^{\gamma_{j}}$, we obtain fundamental solutions of Tsuda's hypergeometric equations \cite{T}
	\begin{gather*}
		\left\{t_{s}(\beta_{s}+\mathcal{D}_{s})\prod_{j=1}^{N}(\alpha_{j}+\mathcal{D})-\mathcal{D}_{s}\prod_{j=1}^{N}(\gamma_{j}-1+\mathcal{D})\right\}y=0,\qquad 1\leq s\leq M,\\
		\{t_{r}(\beta_{r}+\mathcal{D}_{r})\mathcal{D}_{s}-t_{s}(\beta_{s}+\mathcal{D}_{s})\mathcal{D}_{r}\}y=0,\qquad 1\leq r<s\leq M,
	\end{gather*}
	where $\mathcal{D}_{s}=t_{s}\frac{\partial}{\partial t_{s}}$ and $\mathcal{D}=\sum_{s=1}^{M}\mathcal{D}_{s}$.
	Similar to the method of Theorem \ref{thm1}, it is expected that the connection problem of Tsuda's equations will be solved.
	In this case, it must be noted that solutions are multivalued functions.
	The connection formula of Tsuda's hypergeometric function $F_{N+1,M}$, which is a solution of Tsuda's equations, depends on the path on $X=\big\{(t_{1},\dots,t_{M})\in\mathbb{C}^{M};\, t_{i}\neq t_{j},\, i\neq j,\, t_{i}\neq 0,1\big\}$.
	In \cite{Matsu}, a path of connection for solutions of GG system~\cite{GG} was discussed by Matsubara-Heo.
	Thus the connection problem of Tsuda's equations will also be solved by Matsubara's method.
	Note that in the case of $N=1$, some of connection formulas related with~$F_{1}$ were obtained by Olsson~\cite{O}, and related with~$F_{D}$ were obtained by Mimachi \cite{Mjap} from the viewpoint of the Jordan--Pochhammer integral and the intersection theory.
	\item[(ii)]The function is a generalization of the $q$-Lauricella function~$\varphi_{D}$.
	In differential case and two variable case, i.e.,
	\begin{gather*}
		[t_{s}(\beta_{s}+\mathcal{D}_{s})(\alpha+\mathcal{D}_{s}+\mathcal{D}_{r})-\mathcal{D}_{s}(\gamma-1+\mathcal{D}_{s}+\mathcal{D}_{r})]F=0,\qquad (i,j)=(1,2),\, (2,1),
	\end{gather*}
	there are power series solutions of the equation which converge near $(t_{1},t_{2})=(0,0)$, $(0,1)$, $(0,\infty)$, $(1,0)$, $(1,1)$, $(1,\infty)$, $(\infty,0)$, $(\infty,1)$, $(\infty,\infty)$.
	Here we show two solutions as examples:
	\begin{gather*}
		F_{1}\left(\begin{matrix}
			\alpha;\beta_{1},\beta_{2}\\
			\alpha+\beta_{1}+\beta_{2}+1-\gamma
		\end{matrix};1-t_{1},1-t_{2}\right),\\
		t_{1}^{-\beta_{1}}t_{2}^{-\beta_{2}}F_{1}\left(\begin{matrix}
			\beta_{1}+\beta_{2}+1-\gamma;\beta_{1},\beta_{2}\\
			\alpha+\beta_{1}+\beta_{2}+1-\gamma
		\end{matrix};\frac{t_{1}-1}{t_{1}},\frac{t_{2}-1}{t_{2}}\right).
	\end{gather*}
	60 solutions like these series were obtained by Appell \cite{App}, Le Vavasseur \cite{Va}, and other solutions which are given by the Horn's $G_{2}$ function were obtained by Erd\'{e}lyi~\cite{Er}.
	In this paper we get power series solutions for the system $E_{N,M}$, however in the case $N=1$ and $M=2$, the series are $q$-analogs of only a part of Appell, Le Vavasseur, Erd\'{e}lyi's solutions.
	In case of one variable, Hahn \cite{Ha} obtained $q$-analogs of Kummer's 24 solutions.
	In future we hope that other solutions for $E_{N,M}$, which contain $q$-analogs of Appell, Le Vavasseur, Erd\'{e}lyi's solutions will be obtained.
	It is expected that these solutions will be generalizations of Hahn's solutions.
	\item[(iii)]Our fundamental solutions $\bm{u}^{L,\sigma}$ are given by series.
	On the other hand, a solution $\mathcal{F}_{N,M}$ has the Euler-type Jackson integral representation (\ref{JACKSON}).
	It is also expected that other solutions of the $q$-difference system $E_{N,M}$ have the Euler-type Jackson integral representation with suitable domain of integration.
	\item[(iv)]In recent years, the theory of elliptic difference equations has progressed.
	For example, considering discrete isomonodromic deformations of a linear difference system, an elliptic Garnier system which is a generalization of elliptic Painlev\'{e} equation defined by Sakai~\cite{S} was obtained by Ormerod, Rains~\cite{OR} and Yamada~\cite{Y}.
	Also, by using representation theory of the elliptic quantum group $U_{q,p}\big(\widehat{\mathfrak{sl}}_{N}\big)$, an explicit formula for elliptic hypergeometric integral solutions of the face type elliptic $q$-KZ equation was obtained by Konno~\cite{K}.
	We hope that an elliptic analog of the hypergeometric function $\mathcal{F}_{N,M}$ and its related isomonodromic system will be obtained and our result will be extended to the elliptic hypergeometric function.
\end{itemize}

\section*{Acknowledgements}
The author would like to thank Professor Yasuhiko Yamada for useful discussions, valuable suggestions and his encouragement.
He also thanks Professor Saiei-Jaeyeong Matsubara-Heo for useful comments.
He is also grateful to Professor Wayne Rossman for careful reading and corrections in the manuscript.
In addition, he thanks the referees for reading carefully the original manuscript and giving valuable suggestions.

\pdfbookmark[1]{References}{ref}
\LastPageEnding

\end{document}